\documentclass[11pt,review]{elsarticle}
\usepackage[utf8]{inputenc}
\usepackage{graphicx}
\graphicspath{ {./images/} }
\usepackage{mathtools}
\allowdisplaybreaks

\usepackage{srcltx}
\usepackage{eurosym}
\usepackage{amssymb}
\usepackage[all,arc]{xy}
\usepackage{enumerate}
\usepackage{mathrsfs}
\usepackage{pst-node}
\usepackage{amsmath,amsthm}
\usepackage{mathtools}
\usepackage[T1]{fontenc}
\usepackage[doublespacing]{setspace}
\usepackage[titletoc,toc,page]{appendix}

\usepackage[colorlinks,plainpages=true,pdfpagelabels,hypertexnames=true,colorlinks=true,pdfstartview=FitV,linkcolor=blue,citecolor=red,urlcolor=black,pagebackref]{hyperref}
\PassOptionsToPackage{unicode}{hyperref}
\PassOptionsToPackage{naturalnames}{hyperref}

\newcommand*\bb{\mathbb}

\newcommand *\w{^\wedge}

\newcommand*\de{\partial}

\makeatletter
\newcommand{\vo}{\vec{o}\@ifnextchar{^}{\,}{}}
\makeatother
\def\YYint#1#2#3{{\setbox0=\hbox{$#1{#2#3}{\iint}$}
    \vcenter{\hbox{$#2#3$}}\kern-.50\wd0}}
\def\XXint#1#2#3{{\setbox0=\hbox{$#1{#2#3}{\int}$}
    \vcenter{\hbox{$#2#3$}}\kern-.50\wd0}}

\makeatletter
\def\namedlabel#1#2{\begingroup
   \def\@currentlabel{#2}%
   \label{#1}\endgroup
}
\makeatother
\makeatletter
\newcommand{\rmh}[1]{\mathpalette{\raisem@th{#1}}}
\newcommand{\raisem@th}[3]{\hspace*{-1pt}\raisebox{#1}{$#2#3$}}
\makeatother

\newcommand{\descref}[2]{\hyperref[#1]{\textnormal{\textcolor{black}{(}\textcolor{blue}{\bf #2}\textcolor{black}{)}}}}
\newcommand{\dref}[2]{\hyperref[#1]{\textcolor{black}{(}\textcolor{blue}{\bf #2}\textcolor{black}{)}}}
\newcommand\RR{\mathbb{R}}

\newcommand\NN{\mathbb{N}}
\newcommand{\ve}{\varepsilon}

\newcommand{\Om}{\Omega}

\usepackage[ddmmyyyy]{datetime}
\usepackage[margin=2cm]{geometry}
\makeatletter
\g@addto@macro\normalsize{%
  \setlength\abovedisplayskip{2pt}
  \setlength\belowdisplayskip{2pt}
  \setlength\abovedisplayshortskip{4pt}
  \setlength\belowdisplayshortskip{4pt}
}
\numberwithin{equation}{section}
\everymath{\displaystyle}
\usepackage[capitalize,nameinlink]{cleveref}
\crefname{section}{Section}{Sections}
\crefname{subsection}{Subsection}{Subsections}
\crefname{condition}{Condition}{Conditions}
\crefname{hypothesis}{Hypothesis}{Hypothesis}
\crefname{assumption}{Assumption}{Assumptions}
\crefname{lemma}{Lemma}{Lemmas}
\crefname{claim}{Claim}{Claims}
\crefname{remark}{Remark}{Remarks}

\crefformat{equation}{\textup{#2(#1)#3}}
\crefrangeformat{equation}{\textup{#3(#1)#4--#5(#2)#6}}
\crefmultiformat{equation}{\textup{#2(#1)#3}}{ and \textup{#2(#1)#3}}
{, \textup{#2(#1)#3}}{, and \textup{#2(#1)#3}}
\crefrangemultiformat{equation}{\textup{#3(#1)#4--#5(#2)#6}}%
{ and \textup{#3(#1)#4--#5(#2)#6}}{, \textup{#3(#1)#4--#5(#2)#6}}%
{, and \textup{#3(#1)#4--#5(#2)#6}}

\Crefformat{equation}{#2Equation~\textup{(#1)}#3}
\Crefrangeformat{equation}{Equations~\textup{#3(#1)#4--#5(#2)#6}}
\Crefmultiformat{equation}{Equations~\textup{#2(#1)#3}}{ and \textup{#2(#1)#3}}
{, \textup{#2(#1)#3}}{, and \textup{#2(#1)#3}}
\Crefrangemultiformat{equation}{Equations~\textup{#3(#1)#4--#5(#2)#6}}%
{ and \textup{#3(#1)#4--#5(#2)#6}}{, \textup{#3(#1)#4--#5(#2)#6}}%
{, and \textup{#3(#1)#4--#5(#2)#6}}

\crefdefaultlabelformat{#2\textup{#1}#3}

\newtheorem{theorem}{Theorem}[section]
\newtheorem{lemma}[theorem]{Lemma}

\newtheorem{proposition}[theorem]{Proposition}

\newtheorem{definition}[theorem]{Definition}% Use {\rm ...}
\newtheorem{remark}[theorem]{Remark}        % Use {\rm ...}

\numberwithin{equation}{section}

\usepackage{enumitem}
\newlist{steps}{enumerate}{1}
\setlist[steps, 1]{label = \textcolor{Cerulean}{Step \arabic*:}}
\makeatletter
\def\ps@pprintTitle{%
	\let\@oddhead\@empty
	\let\@evenhead\@empty
	\def\@oddfoot{}%
	\let\@evenfoot\@oddfoot}
\makeatother
\DeclarePairedDelimiterX{\inp}[2]{\langle}{\rangle}{#1, #2}
\newcommand{\norm}[1]{\left\lVert#1\right\rVert}

\usepackage{doi,esint}
\begin{document}

\begin{frontmatter}
\title{Existence of variational solutions to doubly nonlinear nonlocal evolution equations via minimizing movements}
\author[myaddress]{Suchandan Ghosh}
\ead{suchandan@tifrbng.res.in}
\author[myaddress]{Dharmendra Kumar}
\ead{dharmendra2020@tifrbng.res.in}
\author[myaddress]{Harsh Prasad}
\ead{harsh@tifrbng.res.in}
\author[myaddress]{Vivek Tewary\corref{mycorrespondingauthor}}
\ead{vivektewary@gmail.com and vivek2020@tifrbng.res.in}

%\tnotetext[thanksfirstauthor]{Supported by the Department of Atomic Energy,  Government of India, under	project no.  12-R\&D-TFR-5.01-0520}

\address[myaddress]{Tata Institute of Fundamental Research, Centre for Applicable Mathematics, Bangalore, Karnataka, 560065, India}
\cortext[mycorrespondingauthor]{Corresponding author}
%\author[1]{Harsh Prasad}
%\author[2]{Vivek Tewary}
%\affil[1,2]{Tata Institute of Fundamental Research, Centre for Applicable Mathematics, Bangalore 560065, Karnataka, India}
\date{\today}
\begin{abstract}
    We prove existence of variational solutions for a class of doubly nonlinear nonlocal evolution equations whose prototype is the double phase equation 
    	\begin{align*}
        \de_t u^m + P.V.\int_{\mathbb{R}^N}& \frac{|u(x,t)-u(y,t)|^{p-2}(u(x,t)-u(y,t))}{|x-y|^{N+ps}}\\
        &+a(x,y)\frac{|u(x,t)-u(y,t)|^{q-2}(u(x,t)-u(y,t))}{|x-y|^{N+qr}} \,dy = 0,\,m>0,\,p>1,\,s,r\in (0,1).
   		\end{align*}
    We make use of the approach of minimizing movements pioneered by DeGiorgi \cite{degiorgiNewProblemsMinimizing1993} and Ambrosio \cite{ambrosioMinimizingMovements1995} and refined by B\"ogelein, Duzaar, Marcellini, and co-authors to study nonlinear parabolic equations with non-standard growth.
\end{abstract}
    \begin{keyword} Nonlocal operators with nonstandard growth, parabolic equations; Parabolic minimizers; Evolutionary variational solutions
    \MSC[2010] 35K51, 35A01, 35A15, 35R11.
    \end{keyword}

\end{frontmatter}
\begin{singlespace}
\tableofcontents
\end{singlespace}

\section{Introduction}

\subsection{The problem}
The notion of variational solutions was introduced by Lichnewsky and Temam in~\cite{lichnewskyPseudosolutionsTimedependentMinimal1978} to study parabolic equations by a variational approach, which has advantages when studying equations with unbalanced growth conditions where Lavrentiev phenomenon can prevent the existence of weak solutions. These methods were greatly refined and expaned upon by the group of B\"ogelein, Duzaar, Marcellini and coauthors in a series of papers such as \cite{bogeleinExistenceEvolutionaryVariational2014,bogeleinParabolicEquationsGrowth2013,bogeleinTimeDependentVariational2015,bogeleinDoublyNonlinearEquations2018,bogeleinExistenceVariationalSolutions2018,bogeleinExistenceEvolutionaryProblems2019,bogeleinExistenceSolutionsDiffusive2021,schatzlerExistenceSingularDoubly2019}. The group of Stefanelli has also conducted extensive work on variational methods in the existence theory of evolution equations starting from his early work on De Giorgi conjecture~\cite{stefanelliGiorgiConjectureElliptic2011} and subsequent works~\cite{akagiDoublyNonlinearEquations2014,scarpaDoublyNonlinearStochastic2020,scarpaStochasticPDEsConvex2021}.

In a recent preprint~\cite{prasadExistenceVariationalSolutions2021a}, the latter two authors extended the framework of variational solutions to parabolic fractional equations with time independent initial and boundary data. As an application, the latter two authors also studied the local boundedness of variational solutions to double phase nonlocal parabolic equations in~\cite{prasadLocalBoundednessVariational2021}.

In the present work, our aim is to study the fractional variant of the work of B\"ogelein et al~\cite{bogeleinDoublyNonlinearEquations2018}. We hope that this would open the path to studying boundedness and H\"older regularity for variational solutions of doubly nonlinear double phase nonlocal parabolic equations. In the recent work~\cite{banerjeeLocalPropertiesSubsolution2021}, regularity of non-negative weak solutions of doubly nonlinear nonlocal parabolic equations was studied.

To be precise, we will study the following equation
\begin{align}\label{maineq}
\de_t b(u) &+ P.V.\int_{\mathbb{R}^N} \frac{\text{D}_{\xi}H(x,y,u(x,t)-u(y,t))}{|x-y|^N} \,dy = 0 \text{ in } \Omega_{\infty}\nonumber\\
u &= u_0 \text{ on } (\mathbb{R}^N\setminus\Omega)\times(0,\infty)\cup {\Omega}\times\{0\}
\end{align}
where $\Omega$ is an open bounded subset of $\bb{R}^N$, and $\Omega_{\infty} = \Omega\times (0,\infty)$. The function $H=H(x,y,\xi)$ satisfies the following structure condition
\begin{align}
    &\label{eq:bound_below_H} H(x,y,\xi) \geq A\left(\frac{|\xi|}{|x-y|^s}\right)^p\\ 
    &\label{eq: cvx_H} H \text{ is a  Caratheodory function which is convex in the variable } \xi.
\end{align}
We assume $1<p<\infty$, $s \in (0,1)$ and $A>0$. The function $b:[0,\infty)\to[0,\infty)$ is a continuous, piecewise $C^1$ function. We may assume without loss of generality that $b(0)=0$, by replacing $b(u)$ with $b(u)-b(0)$. As in~\cite{bogeleinDoublyNonlinearEquations2018}, we make the following assumptions on the function $b$:
\begin{align}
    &\label{eq:bound_below_b} b(u)>0\mbox{ for }u>0.\\ 
    &\label{eq:growth_b} l\leq \frac{u b'(u)}{b(u)}\leq m,\,\mbox{ for } u>0,
\end{align} whenever $b'(u)$ exists and for $m\geq l>0$. In particular, this implies the non-negativity of $b'(u)$ whenever it is defined. 

These structure conditions admit a variety of problems with non-standard growth such as
\begin{itemize}
    \item $H(x,y,\xi) = \left(\frac{|\xi|}{|x-y|^s}\right)^p+a(x,y)\left(\frac{|\xi|}{|x-y|^r}\right)^q$, for $0\leq a(x,y)$, $1<p<q$ and $r,s\in (0,1)$. 
    \item $H(x,y,\xi) = \left(\frac{|\xi|}{|x-y|^s}\right)^p\log\left(1+\left(\frac{|\xi|}{|x-y|^s}\right)\right).$
    \item $H(x,y,\xi) = \left(\frac{|\xi|}{|x-y|^s}\right)^{a+b\sin\left(\log\,\log\left(\frac{|\xi|}{|x-y|^s}\right)\right)}$.
\end{itemize}
In particular, we would like to emphasize that in the double phase case, we obtain existence of variational solutions \textit{without} any restrictions on the gap $q-p$.  

\subsection{Background}

Regularity theory of fractional $p$-Laplace equations and their parabolic counterparts, along with other nonlinear variants has seen a lot of growth in recent years, particulary following the authoritative account of the associated functional frameworks in~\cite{dinezzaHitchhikersGuideFractional2012}. However, a theory for parabolic fractional equations with non-standard growth requires the development of an existence theory which can handle the possibility of the appearance of Lavrentiev phenomenon. The variational framework introduced in~\cite{bogeleinExistenceEvolutionaryVariational2014} proves to be particularly well-suited for such a study. 

On the other hand, the existence of solutions for doubly nonlinear equations for local equations found a definitive treatment in~\cite{altQuasilinearEllipticparabolicDifferential1983}. A variational approach was introduced in~\cite{akagiDoublyNonlinearEquations2014} as well as~\cite{bogeleinDoublyNonlinearEquations2018}, which differ in many specifics as described in the latter article, particularly as it eschews any growth condition from above. Other works tend to be based on the Galerkin method. Some recent works studying existence of solutions to nonlocal parabolic equations, including doubly nonlinear nonlocal equations, are~\cite{giacomoniExistenceStabilizationResults2019,giacomoniExistenceGlobalBehavior2021,liaoGlobalExistenceBlowup2020,liGlobalExistenceAsymptotic2021}.

The regularity theory of $p,q$ growth problems was started by Marcellini in a series of novel papers~\cite{marcelliniRegularityMinimizersIntegrals1989,marcelliniRegularityExistenceSolutions1991,marcelliniRegularityEllipticEquations1993,marcelliniRegularityScalarVariational1996}. There is a large body of work dealing with problems of $(p,q)$-growth as well as other nonstandard growth problems, for which we point to the surveys in \cite{marcelliniRegularityGeneralGrowth2020,mingioneRecentDevelopmentsProblems2021}.

Regarding fractional $p$-Laplace equation, in the elliptic case, regularity theory of fractional $p$-Laplace equations has been studied extensively. An early existence result can be found in~\cite{ishiiClassIntegralEquations2010}. Local boundedness and H\"older regularity in the framework of De Giorgi-Nash-Moser theory was worked out in~\cite{dicastroLocalBehaviorFractional2016} and \cite{cozziRegularityResultsHarnack2017}. Moreover, explicit higher regularity of the gradient is obtained in~\cite{brascoHigherSobolevRegularity2017}. On the other hand, explicit H\"older regularity of the solutions is obtained in \cite{brascoHigherHolderRegularity2018}. Higher integrability by a nonlocal version of Gehring's Lemma was proved in~\cite{kuusiFractionalGehringLemma2014}. Other works of interest are \cite{nowakHsRegularityTheory2020,nowakRegularityTheoryNonlocal2021,nowakImprovedSobolevRegularity2021,nowakHigherHolderRegularity2021}. For equations of nonstandard growth, the relevant works are \cite{scottSelfimprovingInequalitiesBounded2020,byunOlderRegularityWeak2021,chakerRegularitySolutionsAnisotropic2020,chakerRegularityEstimatesFractional2021,chakerNonlocalOperatorsSingular2020,chakerLocalRegularityNonlocal2021,chakerRegularityNonlocalProblems2021,giacomoniOlderRegularityResults2021,byunLocalOlderContinuity2021}. For the case of linear equations, i.e., $p=2$, we refer to \cite{caffarelliDriftDiffusionEquations2010,caffarelliRegularityTheoryParabolic2011,laraRegularitySolutionsNon2014,chang-laraRegularitySolutionsNonlocal2014}.

In the case of parabolic counterparts of the fractional $p$-Laplace equations, local boundedness was proved in~\cite{stromqvistLocalBoundednessSolutions2019}. Local boundedness and H\"older regularity has been proved in~\cite{dingLocalBoundednessHolder2021}. Explicit H\"older regularity has been obtained in~\cite{brascoContinuitySolutionsNonlinear2021}. The latter two authors have proved local boundedness estimates for the double phase nonlocal parabolic equation in~\cite{prasadLocalBoundednessVariational2021}.

\subsection{Framework}

In order to define the notion of weak solution, we need to set up the Orlicz spaces related to the function $b$. For details, we refer to the article of B\"ogelein et al~\cite{bogeleinDoublyNonlinearEquations2018} where these function spaces have been defined. 

Let us denote by $\phi$ the primitive of the function $b$, i.e., 
\begin{align}
    \label{primitive}
    \phi(u):=\int_0^u\,b(s)\,ds,\mbox{ for all }u\geq 0.
\end{align}

The function $\phi$ is $C^1$ and convex since $b$ is an increasing function. Moroever $\phi(0)=0$. We define the convex conjugate of $\phi$ by 
\begin{align}
    \label{conjconv}
    \phi^*(x):=\sup_{v\geq 0}(xv-\phi(v)),\mbox{ for all }x\geq 0.
\end{align}

Due to the convexity of $\phi$, we have
\begin{align}
    \phi^*(b(u))=b(u)u-\phi(u)\mbox{ for all }u\geq 0.
\end{align}

Further, the Fenchel inequality holds
\begin{align}
    \label{fenchel}
    uv\leq \phi(u) + \phi^*(v)\mbox{ for all }u,v\geq 0.
\end{align}

Define the boundary terms 
\begin{align}
    \label{boundary1}
    \mathfrak{b}[u,v]&:=\phi(v)-\phi(u)-b(u)(v-u)\nonumber\\
                    & = \phi(v)+\phi^*(b(u))-b(u)v\mbox{ for all }u,v\geq 0
\end{align} where the second equality follows from the definition of the convex conjugate.

Also define the boundary integral
\begin{align}
    \label{boundary2}
    \mathfrak{B}[u,v]=\int_{\Om} \mathfrak{b}[u,v]\,dx\mbox{ for all }u,v:\Omega\to[0,\infty).
\end{align}

For a domain $A$ in $\mathbb{R}^d$, $d\in\NN$, define the Orlicz space
\begin{align}
    \label{orlicz}
    L^{\phi}(A):=\left\{ v:A\to\RR\mbox{ measurable }: \int_A \phi(|v|)\,dx<\infty \right\}.
\end{align}

This definition is well-defined particularly since $\phi, \phi^*$ satisfy the doubling condition or the $\Delta_2$-condition. This will be a consequence of some technical lemmas which are proved in ~\cite{bogeleinDoublyNonlinearEquations2018} and which we will list in  a subsequent section. The Orlicz space $L^\phi(A)$ is equipped with the norm
\begin{align}
    \label{orlicznorm}
    ||v||_{L^\phi(A)}=\sup\left\{ \left|\int_A\,vw\,dx\right| : \int_A \phi^*(w)\leq 1\right\}.
\end{align}

More details on this restricted class of Orlicz spaces may be found in ~\cite{adamsSobolevSpaces2003}.

\subsection{Definition}

\begin{definition}
    Let $\Om$ be an open bounded subset of $\RR^N$. Suppose that $H$ satisfies the assumptions \cref{eq:bound_below_H}, \cref{eq: cvx_H} and $b$ satisfies the assumptions \cref{eq:growth_b} and \cref{eq:bound_below_b}. Let the time-independent Cauchy-Dirichlet data $u_0:\RR^N\to [0,\infty)$ satisfy
    \begin{align}\label{datahypo}
        u_0\in W^{s,p}(\mathbb{R}^N), u_{|\Omega}\in L^\phi(\Omega)\mbox{ and }\sup_{t\in (0,T)}\iint\limits_{\mathbb{R}^N\times\mathbb{R}^N}\frac{H(x,y,u_0(x)-u_0(y))}{|x-y|^N}\,dx\,dy<\infty.
    \end{align}
     By a variational solution to~\cref{maineq} we mean a function $u:(0,T)\times\RR^N\to [0,\infty)$ such that \[u\in L^p(0,T;W^{s,p}(\mathbb{R}^N))\cap C^0(0,T;L^\phi(\Om)), u-u_0\in L^p(0,T;W^{s,p}_0(\Om))\] and
     \begin{align}
         \label{defvar}
         \int_{0}^{\tau} \int_{\Om}{\de_t v}\,{(b(v)-b(u))}\,dx\,dt & + \int_{0}^{\tau}\iint\limits_{\RR^N\times\RR^N} \frac{H(x,y,v(x,t)-v(y,t))-H(x,y,u(x,t)-u(y,t))}{|x-y|^N}\,dx\,dy\,dt\nonumber\\
         &\geq \mathfrak{B}[u(\tau),v(\tau)]-\mathfrak{B}[u_0,v(0)],
     \end{align} for any $\tau\in [0,T]$ and for all $v:(0,T)\times\RR^N\to [0,\infty)$ such that $v\in L^p(0,T;W^{s,p}(\mathbb{R}^N))$ and $\de_t v\in L^\phi(0,T;\Om)$ such that $v-u_0\in L^p(0,T;W_0^{s,p}(\Om))$ and $v(0)\in L^\phi(\Om)$.
\end{definition}
\begin{remark}
In our definition of variational solution, both the solution and the comparison map have to match on $\Omega^c\times(0,T)$ and since we have assumed that the data on $\Omega^c$ is in $W^{s,p}(\bb{R}^N)$ we may cancel integrals of $H$ over $\Omega^c\times\Omega^c$ on both sides to obtain the following equivalent form of the variational inequality: 
\begin{align*}
    \int_{0}^{\tau} \int_{\Om}{\de_t v}\,{(b(v)-b(u))}\,dx\,dt & + \int_{0}^{\tau}\iint\limits_{C_\Om} \frac{H(x,y,v(x,t)-v(y,t))-H(x,y,u(x,t)-u(y,t))}{|x-y|^N}\,dx\,dy\,dt\nonumber\\
         &\geq \mathfrak{B}[u(\tau),v(\tau)]-\mathfrak{B}[u_0,v(0)]
\end{align*} where $C_{\Omega} = (\Omega^c\times\Omega^c)^c$. 
\end{remark}
\begin{remark}
Let us also mention that in the double phase case i.e. when $H$ is of the form
\[
H(x,y,\xi) = \left(\frac{|\xi|}{|x-y|^s}\right)^p+a(x,y)\left(\frac{|\xi|}{|x-y|^r}\right)^q
\]
the condition on initial data ~\cref{datahypo} is satisfied if $u_0\in W^{s,p}(\RR^N)$, $u_{|\Om}\in L^2(\Om)$ and:
\[
\iint\limits_{\mathbb{R}^N\times\mathbb{R}^N}\frac{|u_0(x)-u_0(y)|^p}{|x-y|^{N+sp}}+ a(x,y)\frac{|u_0(x)-u_0(y)|^q}{|x-y|^{N+rq}}\,dx\,dy<\infty.
\]
\end{remark}
\subsection{Main Results}

The following existence theorem is the main result of the paper. 

\begin{theorem}{\textbf{(Existence of variational solutions)}}
    \label{mainthm:e}
    Let $\Om$ be an open bounded subset of $\RR^N$. Suppose that $H$ satisfies the assumptions~\cref{eq:bound_below_H} and~\cref{eq: cvx_H} and $b$ satisfies the assumptions \cref{eq:growth_b} and \cref{eq:bound_below_b}. Let the time-independent Cauchy-Dirichlet data $u_0$ satisfy~\cref{datahypo}. Then, there exists a variational solution to~\cref{maineq} in the sense of~\cref{defvar}, with $\de_t\sqrt{\phi(u)}\in L^2(\Om_T)$. The initial data is attained in the $L^\phi$ sense.
\end{theorem}

\begin{remark}
    Time-dependent functions $H$ can be studied in the manner of \cite{schatzlerExistenceVariationalSolutions2017}. As a first step, one considers $H$ that have a $p$ growth condition from above and Lipschitz regularity in the time-variable. To consider more general $H$ with only measurability in the time-variable, one needs to regularize $H$ and then pass to the limit in the regularizing parameter. This last step, however, is not straightforward as it involves composition of weak limits and nonlinearity.
\end{remark}

\begin{remark}
    As remarked in \cite{bogeleinDoublyNonlinearEquations2018}, uniqueness is a delicate issue. However, for $H$ with $p$-growth condition and $b$ Lipschitz continuous, uniqueness may be proved in a manner similar to \cite{altQuasilinearEllipticparabolicDifferential1983}.
\end{remark}

\begin{remark}
    Existence in unbounded domains $\Om$ is also proved in \cite{bogeleinDoublyNonlinearEquations2018}. This requires construction of a suitable function space which has Poincar\'e inequality built in. This may be achieved by taking completion of $C_0^\infty(\Om)$ in the Gagliardo seminorm $[\cdot]_{s,p}$.
\end{remark}

\section{Preliminaries}
\subsection{Notations}
We begin by collecting the standard notation that will be used throughout the paper:
\begin{itemize}
\item We shall denote $N$ to be the space dimension. We shall denote by $z=(x,t)$ a point in $ \RR^N\times (0,T)$.  
\item We shall alternately use $\dfrac{\partial f}{\partial t},\partial_t f,f'$ to denote the time derivative of f.
 \item Let $\Omega$ be an open bounded domain in $\mathbb{R}^N$ with boundary $\partial \Omega$ and for $0<T\leq\infty$,  let $\Omega_T\coloneqq \Omega\times (0,T)$. 
 %\item We shall use the notation
%\begin{align*}
%	B_{\rho}(x_0)=\{x\in\RR^N:|x-x_0|<\rho\},\\
%	\overline{B}_{\rho}(x_0)=\{x\in\RR^N:|x-x_0|\leq\rho\}.
%\end{align*} 
%\item The maximum of two numbers $a$ and $b$ will be denoted by $a\wedge b\coloneqq \max(a,b)$. 
\item Integration with respect to either space or time only will be denoted by a single integral $\int$ whereas integration on $\Om\times\Om$ or $\RR^N\times\RR^N$ will be denoted by a double integral $\iint$.
%\item The notation $a \lesssim b$ is shorthand for $a\leq C b$ where $C$ is a universal constant which only depends on the dimension $N$, exponents $p$, $q$, and the numbers $A$ and $s$. 
\end{itemize}

\subsection{Sobolev spaces}
Let $1<p<\infty$, we denote by $p'=p/(p-1)$ the conjugate exponent of $p$. Let $\Om$ be an open subset of $\RR^N$. We define the {\it Sobolev-Slobodeki\u i} space, which is the fractional analogue of Sobolev spaces.
\begin{align*}
    W^{s,p}(\Om)=\left\{ \psi\in L^p(\Omega): [\psi]_{W^{s,p}(\Om)}<\infty \right\}, s\in (0,1),
\end{align*} where the seminorm $[\cdot]_{W^{s,p}(\Om)}$ is defined by 
\begin{align*}
    [\psi]_{W^{s,p}(\Om)}=\left(\iint\limits_{\Om\times\Om} \frac{|\psi(x)-\psi(y)|^p}{|x-y|^{N+ps}}\,dx\,dy\right)^{\frac 1p}.
\end{align*}
The space when endowed with the norm $\norm{\psi}_{W^{s,p}(\Om)}=\norm{\psi}_{L^p(\Om)}+[\psi]_{W^{s,p}(\Om)}$ becomes a Banach space. The space $W^{s,p}_0(\Om)$ is the subspace of $W^{s,p}(\RR^N)$ consisting of functions that vanish outside $\Om$.

Let $I$ be an interval and let $V$ be a separable, reflexive Banach space, endowed with a norm $\norm{\cdot}_V$. We denote by $V^*$ its topological dual space. Let $v$ be a mapping such that for a.e. $t\in I$, $v(t)\in V$. If the function $t\mapsto \norm{v(t)}_V$ is measurable on $I$, then $v$ is said to belong to the Banach space $L^p(I;V)$ if $\int_I\norm{v(t)}_V^p\,dt<\infty$. It is well known that the dual space $L^p(I;V)^*$ can be characterized as $L^{p'}(I;V^*)$.

\subsection{Mollification in time}
Throughout the paper, we will use the following mollification in time. Let $\Omega$ be an open subset of $\mathbb{R}^N$. For $T>0$, $v\in L^1(\Omega_T)$, $v_0\in L^1(\Omega)$ and $h\in (0,T]$, we define
\begin{align}
    \label{timemollify}
    [v]_h(\cdot,t)=e^{-\frac{t}{h}}v_0 + \frac{1}{h}\int_0^t e^{\frac{s-t}{h}}v(\cdot,s)\,ds,
\end{align} for $t\in [0,T]$.
The convergence properties of mollified functions have been collected in~\ref{Molli}.

\subsection{Auxiliary Results}
We collect the following standard results which will be used in the course of the paper. We begin with a general result on convex minimization.

\begin{proposition}(\cite[Theorem~2.3]{rindlerCalculusVariations2018})\label{mintheorem}
    Let $X$ be a closed affine subset of a reflexive Banach space and let $\mathcal{F}: X\to (-\infty,\infty]$. Assume the following:
    \begin{enumerate}
        \item For all $\Lambda\in\RR$, the sublevel set $\{x\in X: \mathcal{F}[u]<\Lambda\}$ is sequentially weakly precompact, i.e., if for a sequence $(u_j)\subset X$, $\mathcal{F}[u_j]<\Lambda$, then $u_j$ has a weakly convergent subsequence.
        \item For all sequences $(u_j)\subset X$ with $u_j\rightharpoonup u$ in $X$-weak, it holds that
        \begin{align*}
            \mathcal{F}[u] \leq \liminf\limits_{j\to\infty} \mathcal{F}[u_j].
        \end{align*}
    \end{enumerate}
    Then, the problem minimization problem: Minimize $\mathcal{F}[u]$ over all $u\in X$ has a solution.
\end{proposition}

We will need the following general result on weak lower semicontinuity of functionals.

\begin{proposition}(\cite[Corollary~3.9]{brezisFunctionalAnalysisSobolev2011})
    \label{lowsemcont} 
    Assume that $\phi:E\to (-\infty,\infty]$ is convex and lower semicontinuous in the strong topology. Then $\phi$ is lower semicontinuous in the weak topology.
\end{proposition}

We have the following Sobolev-type inequality~\cite[Theorem~6.5]{dinezzaHitchhikersGuideFractional2012}.

\begin{theorem}[\cite{dinezzaHitchhikersGuideFractional2012}]
    Let $s\in(0,1)$ and $1\leq p<\infty$, $sp<N$ and let $\kappa^*=\frac{N}{N-sp}$, then for any $g\in W^{s,p}(\mathbb{R}^N)$ and $\kappa\in [1,\kappa^*]$, we have
    \begin{align}\label{sobolev1}
        \norm{g}_{L^{\kappa p}}^p\leq C\iint\limits_{\mathbb{R}^N\times\mathbb{R}^N}\frac{|g(x)-g(y)|^p}{|x-y|^{N+sp}}\,dx\,dy.
    \end{align}. If $g\in W^{s,p}(\Om)$ and $\Om$ is an extension domain, then
    \begin{align}
        \label{sobolev2}
        ||g||_{L^{\kappa p}(\Om)}^p\leq C\norm{g}_{W^{s,p}(\Om)}.
    \end{align} If $sp=N$, then \cref{sobolev1} holds for all $\kappa\in [1,\infty)$, whereas if $sp>N$, then \cref{sobolev2} holds for all $\kappa\in[1,\infty]$.
\end{theorem}

We will require the compact embedding as follows.

\begin{proposition}(\cite[Proposition~2.1]{hanCompactSobolevSlobodeckijEmbeddings2022})
    \label{compactembed}
    Assume $N\geq 2$, $1\leq p\leq \infty$ and $0<s<1$. Let $\Om$ be a bounded extension domain. When $sp<N$, the embedding $W^{s,p}(\Om)\mapsto L^r(\Om)$ is compact for $r\in [1,p_*)$, and when $sp\geq N$, the same embedding is compact for $r\in [1,\infty)$. The theorem also holds for $W^{s,p}_0(\Om)$ for any bounded domain in $\RR^N$.
\end{proposition}

We end this section with a Poincar\'e inequality for Gagliardo seminorms.

\begin{proposition}(\cite[Lemma 2.4]{brascoFractionalCheegerProblem2014})
	\label{poincare}
	Let $1\leq p<\infty$, $s\in (0,1)$ and $\Omega$ is an open and bounded set in $\RR^N$. Then, for every $u\in W^{s,p}_0(\Om)$, it holds that
	\begin{align*}
		||u||_{L^p(\Om)}^p\leq C(N,s,p,\Om)\,[u]_{W^{s,p}(\RR^N)}^p
	\end{align*}
\end{proposition}

\subsection{Technical Lemmas}

In this subsection, we collect some technical lemmas regarding the functions $b$ and $\phi$ as proved in \cite{bogeleinDoublyNonlinearEquations2018}. The inequalities hold for those arguments for which $b'$ is defined.

\begin{lemma}(\cite[Lemma 2.1]{bogeleinDoublyNonlinearEquations2018})
    \label{techlem1}
    For $b:[0,\infty)\to[0\infty)$ which is continuous and piecewise $C^1$ and satisfies \cref{eq:growth_b}, the following estimates are true for every $\lambda>1$ and $u>0$:
    \begin{enumerate}
        \item $\lambda^l b(u)\leq b(\lambda u)\leq \lambda^m b(u).$
        \item $b(1)\min\{u^l,u^m\}\leq b(u)\leq b(1)\max\{u^l,u^m\}$.
        \item $\frac{l}{m}\lambda^{l-1}b'(u)\leq b'(\lambda u)\leq \frac{m}{l}\lambda^{m-1}b'(u)$.
        \item $\lambda^{l+1}\phi(u)\leq \phi(\lambda u)\leq \lambda^{m+1}\phi(u)$.
        \item $\lambda^{\frac{m+1}{m}}\phi^*(u)\leq \phi^*(\lambda u)\leq \lambda^{\frac{l+1}{l}}\phi^*(u)$.
    \end{enumerate}
\end{lemma}

\begin{lemma}(\cite[Lemma 2.3]{bogeleinDoublyNonlinearEquations2018})
    \label{techlem2}
    With the standard assumptions on $b$, we have 
    \[\frac{1}{m+1}ub(u)\leq \phi(u)\leq \frac{1}{l}\phi^*(b(u))\leq \frac{m}{l(m+1)}ub(u)\] for any $u\geq 0$.
\end{lemma}

\begin{lemma}(\cite[Lemma 2.4]{bogeleinDoublyNonlinearEquations2018})
    \label{techlem3}
    With the standard assumptions on $b$, we have for all $u,v\geq 0$ the following estimates
    \[\phi^*(b(u))\leq 2\mathfrak{b}[u,v]+2^{m+2}\phi(v),\mbox{ and}\] 
    \[\phi(v)\leq 2\mathfrak{b}[u,v]+2^{2+\frac{1}{l}}\phi^*(b(u)).\]
\end{lemma}

\begin{lemma}(\cite[Lemma 2.5]{bogeleinDoublyNonlinearEquations2018})
    \label{techlem4}
    With the standard assumptions on $b$, there exists a constant $c=c(m,l)$ such that for all $u,v\geq 0$, the following estimates are true:
    \begin{align*}
        \mathfrak{b}[u,v]&\leq (b(v)-b(u))(v-u)\\
                         &\leq |\sqrt{vb(v)}-\sqrt{ub(u)}|^2\\
                         &\leq c|\sqrt{\phi(v)}-\sqrt{\phi(u)}|^2\leq c^2\mathfrak{b}[u,v].
    \end{align*}
\end{lemma}

\begin{lemma}(\cite[Lemma 2.6]{bogeleinDoublyNonlinearEquations2018})
    \label{convergence1} 
    Let the standard assumptions on $b$ hold. Let $(u_i)_{i\in\NN}$ be a sequence in $L^\phi(\Om)$ and $u_0\in L^\phi(\Om)$ such that $\phi(u_i)\to\phi(u_0)$ in $L^1(\Om)$. Then, there holds $u_i\to u_0$ in $L^\phi(\Om)$.
\end{lemma}

\begin{lemma}(\cite[Lemma 2.7]{bogeleinDoublyNonlinearEquations2018})
    \label{convergence2} 
    Let the standard assumptions on $b$ hold. Let $v\in L^1(\Om_T)$ is given with $\de_t v\in L^\phi(\Om)$. Then we have $v\in C^0([0,T];L^\phi(\Om))$.
\end{lemma}

\subsection{A lower bound on the nonlinear operator}

In this subsection, we derive a lower bound on the nonlinear and nonlocal operator by the use of Poincar\'e inequality (\cite[Lemma 2.4]{brascoFractionalCheegerProblem2014}).

Let $u\in W_{u_0}^{s,p}(\Om)$, then $u-u_0\in W^{s,p}_0(\Om)$. Therefore, by Poincar\'e inequality in the space $W_0^{s,p}(\Om)$, we get for any $t\in [0,T]$,

\begin{align*}
    ||u-u_0||_{L^p(\Om)}^p&\leq C [u-u_0]^p_{W^{s,p}(\RR^N)}\\
                          &\leq C [u]_{W^{s,p}(\RR^N)}^p + C [u_0]_{W^{s,p}(\RR^N)}^p\\
                          &\leq C \iint\limits_{\RR^N\times\RR^N} \frac{H(x,y,u(x,t)-u(y,t))}{|x-y|^N}\,dx\,dy+ C [u_0]_{W^{s,p}(\RR^N)}^p.
\end{align*}

As a result, we obtain
\begin{align}\label{coercivity1}
    ||u||_{L^p(\RR^N)}^p &\leq ||u-u_0||_{L^p(\Om)}^p+||u_0||_{(L^p(\RR^N))}^p\nonumber\\
    &\leq C_1 \iint\limits_{\RR^N\times\RR^N} \frac{H(x,y,u(x,t)-u(y,t))}{|x-y|^N}\,dx\,dy+ C_2 ||u_0||_{W^{s,p}(\RR^N)}^p,
\end{align} for some positive constants $C_1$ and $C_2$.

Once again, due to \cref{eq:bound_below_H}, we get
\begin{align}\label{coercivity2}
	||u||_{W^{s,p}(\RR^N)}^p \leq C_1 \iint\limits_{\RR^N\times\RR^N} \frac{H(x,y,u(x,t)-u(y,t))}{|x-y|^N}\,dx\,dy+ C_2 ||u_0||_{W^{s,p}(\RR^N)}^p.
\end{align}

\subsection{Regarding initial condition}

In this subsection, we prove that variational solutions satisfy the initial condition in the sense of $L^\phi(\Om)$.

\begin{lemma}
    Any variational solution as defined in \cref{defvar} satisfies the initial condition $u(0)=u_0$ in the $L^\phi(\Om)$ sense, i.e.,
    \begin{align*}
        \lim_{\tau\downarrow 0}||u(\tau)-u(0)||_{L^\phi(\Om)}=0
    \end{align*}
\end{lemma}

\begin{proof}
    Taking the time-independent extension of $v(x,t)=u_0(x)$ as a comparision map in the definition of variational solutions \cref{defvar}, we get 
    \begin{align*}
        \mathfrak{B}[u(\tau),u_0]\leq \tau\sup_{t\in [0,T]}\iint\limits_{\RR^N\times\RR^N} \frac{H(x,y,u_0(x)-u_0(y))}{|x-y|^N}\,dx\,dy\to 0\mbox{ as }\tau\to 0.
    \end{align*}
    On the other hand
    \begin{align*}
        \int_{\Om} |\phi(u(\tau))-\phi(u_0)|\,dx&\leq \left(\int_{\Om} |\sqrt{\phi(u(\tau))}-\sqrt{\phi(u_0)}|^2\,dx\right)^{1/2}\left(\int_{\Om} |\sqrt{\phi(u(\tau))}+\sqrt{\phi(u_0)}|^2\,dx\right)^{1/2}\\
        &\leq C\left(\mathfrak{B}[u(\tau),u_0]\right)^{1/2}\to 0\mbox{ as }\tau\to 0.
    \end{align*} where in the last inequality, we have used \cref{techlem4}. As a result, we have $\phi(u(\tau))\to\phi(u_0)$ in $L^1(\Om)$. Therefore, by an application of \cref{convergence1}, we get $u(\tau)\to u_0$ in $L^\phi(\Om)$.
\end{proof}

\subsection{A discrete integration by parts formula}

We state a discrete integration by parts formula as proved in \cite{bogeleinDoublyNonlinearEquations2018}. For $h\in\RR\setminus\{0\}$, denote the difference quotient of a function $w$ with respect to time by $\Delta_h w$ and define as
\begin{align*}
    \Delta_h w(t):=\frac{1}{h}(w(t+h)-w(t)).
\end{align*}

\begin{lemma}(\cite[Lemma 2.10]{bogeleinDoublyNonlinearEquations2018})
    \label{discreteintbyparts} Let $h\in (0,1]$, $u,v\in L^\phi(\Om\times(-h,T+h))$ be two non-negative functions. Then, the following integration by parts formula is valid
    \begin{align*}
        \iint\limits_{\Om_T} \Delta_{-h}b(u)(v-u)\,dx\,dt&\leq \iint\limits_{\Om_T} \Delta_h v(b(v)-b(u))\,dx\,dt\\
        &\quad-\frac{1}{h}\iint\limits_{\Om\times(T-h,T)}\mathfrak{b}[u,v(t+h)]\,dx\,dt\\
        &\quad\quad+\frac{1}{h}\iint\limits_{\Om\times(-h,0)}\mathfrak{b}[u,v]\,dx\,dt+\delta_1(h)+\delta_2(h),
    \end{align*}
    where
    \begin{align*}
        \delta_1(h)&:=\frac{1}{h}\iint\limits_{\Om_T}\mathfrak{b}[v(t),v(t+h)]\,dx\,dt,\\
        \delta_2(h)&:=\iint\limits_{\Om\times(-h,0)} \Delta_h v(b(v(t+h))-b(u))\,dx\,dt.
    \end{align*}
    If, in addition, $\de_t v\in L^\phi(\Om\times(-h_0,T+h_0))$ for some $h_0>0$, then 
    \begin{align}
        \lim_{h\downarrow 0}\delta_1(h)=0\mbox{ and }\lim_{h\downarrow 0}\delta_2(h)=0.
    \end{align}
\end{lemma}

\section{Compactness Theorem}

We shall require the following compactness theorem in order to identify the limit function of the approximating minimizers. It is the fractional analogue of \cite[Proposition 3.1]{bogeleinDoublyNonlinearEquations2018}.

\begin{proposition}
    \label{compactnesslemma}
    Let $\Om\subseteq \RR^N$ be a bounded domain, $p>1$, $T>0$ and $k\in\NN$. Suppose that $h_k:=\dfrac{T}{k}$ piecewise constant functions $u^{(k)}:\RR^N\times (-h_k,T]\to\RR$ are defined by 
    \begin{align*}
        u^{(k)}(\cdot,t):=u^{(k)}_i\mbox{ for }t\in ((i-1)h_k,\,ih_k)\mbox{ with }i=0,1,\ldots,k,
    \end{align*} where $u_i^{(k)}\in L^\phi(\Om)\cap W^{s,p}_{u_0}(\Om)$. Suppose further that there exists a constant $M>0$ such that the following energy estimate and continuity estimate hold true for all $k\in\NN$:
    \begin{align}\label{energybound1}
        \max\limits_{i\in\{0,1,\ldots,k\}}\left(\int_{\Omega} \phi(u^{k}_i)\,dx + \iint\limits_{\RR^N\times\RR^N}\frac{|u_i^{(k)}(x)-u_i^{(k)}(y)|^p}{|x-y|^{N+ps}}\,dx\,dy\right)\leq M, \mbox{ and }
    \end{align} 
    \begin{align}\label{continuitybound}
        \frac{1}{h_k}\sum_{i=1}^k \int_\Om \left| \sqrt{\phi(u_i^{(k)})}-\sqrt{\phi(u_{i-1}^{(k)})} \right|^2\,dx \leq M.
    \end{align} Moreover, $u^{(k)}\rightharpoonup u$ in $L^p(0,T;W^{s,p}_{u_0}(\Om))$-weak as $k\to\infty$. Then, there exists a subsequence such that as $k\to\infty$, we have
    \begin{align*}
        &\sqrt{\phi(u^{(k)})}\to\sqrt{\phi(u)} \mbox{ in } L^1(\Om_T)-\mbox{strong},\\
        &u^{(k)}\to u\mbox{ a.e. in }(0,T)\times\RR^N.
    \end{align*}
\end{proposition}

\begin{proof}
    Since we are working with subsequences, we may always begin with a subsequence where $k=2^{\beta}$ for $\beta\in\NN$. Then, for $\alpha\in\NN$, we define another piecewise constant function where each piece is of length $\ve_{\alpha}:=\frac{T}{2^\alpha}$
    \begin{align*}
        w_{k,\alpha}(t):=\sum_{j=1}^{2^{\alpha}} u^{(k)}((j-1)\ve_\alpha) \chi_{((j-1)\ve_\alpha,j\ve_\alpha]}(t).
    \end{align*}
    Then, depending on whether $\alpha\leq \beta$ or $\beta\leq\alpha$, respectively $\ve_\alpha$ is a multiple of $h_k$ or $h_k$ is a multiple of $\ve_\alpha$. For $i\in\{1,2,\ldots,2^{\beta}\}$, define 
    \begin{align*}
        i_{-}:=\left\lfloor \frac{h_k(i-1)}{\ve_\alpha} \right\rfloor \frac{\ve_\alpha}{h_k}+1 = \left\lfloor \frac{i-1}{2^{\beta-\alpha}}\right\rfloor 2^{\beta-\alpha}+1.
    \end{align*}

    Now, we claim that for a fixed $\alpha\in \NN$. there is a subsequence of $k=2^\beta$ labelled $k^{(\alpha)}$ for which 
    $\sqrt{\phi(w_{k^{(\alpha)},\alpha})}$ converges in $L^1(\Om_T)$ as $k^{(\alpha)}\to \infty$. For the proof, observe that by \cref{energybound1} and Poincar\'e inequality in $W^{s,p}_0(\Om)$, the sequence $(u^{(k)}((j-1)\ve_\alpha)-u_0)$ is uniformly bounded in $W^{s,p}_0(\Om)$. Therefore, by the fractional version of Rellich's compactness theorem (\cref{compactembed}), there is a subsequence for which $(u^{(k)}((j-1)\ve_\alpha)-u_0)$ converges in $L^p(\Om)$, and for a further subsequence, $(u^{(k)}((j-1)\ve_\alpha))$ converges pointwise a.e. in $\RR^N$ for any $j\in\{1,2,\ldots,2^\alpha\}$. As a consequence, $(w_{k,\alpha})$, whose summands are exactly $(u^{(k)}((j-1)\ve_\alpha))$, converges pointwise a.e. in $\RR^N$ along a subsequence labelled $k^{(\alpha)}$. As a result, by a standard application of dominated convergence theorem (which will use the uniform bound from \cref{energybound1}), we will get convergence of $\sqrt{\phi(w_{k,\alpha})}$ in $L^1(\Om_T)$ along the subsequence  $k^{(\alpha)}$.

    Hence we obtain convergence along a subsequence for each $\alpha\in\NN$.  Therefore, we can extract a diagonal subsequence $\varkappa$ of $k$ such that 
    \begin{align}
        \label{l1convergence}
        \sqrt{\phi(w_{k,\alpha})}\mbox{ converges in }L^1(\Om_T) \mbox{ as } \varkappa\to\infty\mbox{ for any }\alpha\in\NN.
    \end{align}

    Now, we wish to prove $L^1(\Om_T)$ convergence along a subsequence of the original sequence $\sqrt{\phi(u^{(k)})}$. To achieve this, we start with the following estimate, which is a consequence of the continuity estimate \cref{continuitybound}:
    \begin{align}\label{continuitybound2}
        \iint\limits_{\Om_T} \left| \sqrt{\phi(u^{(k)})}-\sqrt{\phi(w_{k,\alpha})} \right|\,dx\,dt & = h_k \sum_{i=1}^{2^\beta} \int\limits_\Om \left| \sqrt{\phi(u^{(k)}_i)}-\sqrt{\phi(u^{(k)}_{i_{-}})} \right|\,dx\nonumber\\
        & \leq h_k \sum_{i=1}^{2^\beta} \sum_{j=i_{-}+1}^i \int\limits_\Om \left| \sqrt{\phi(u^{(k)}_j)}-\sqrt{\phi(u^{(k)}_{j-1})} \right|\,dx\nonumber\\
        & \leq \ve_\alpha \sum_{j=1}^{2^\beta} \int\limits_\Om \left| \sqrt{\phi(u^{(k)}_j)}-\sqrt{\phi(u^{(k)}_{j-1})} \right|\,dx\nonumber\\
        & \leq \ve_\alpha |\Om|^{1/2} \sum_{j=1}^{2^\beta} \left(\int\limits_\Om \left| \sqrt{\phi(u^{(k)}_j)}-\sqrt{\phi(u^{(k)}_{j-1})} \right|^2\,dx\right)^{1/2}\nonumber\\
        & \leq \ve_\alpha \sqrt{2^\beta|\Om|}  \left(\sum_{j=1}^{2^\beta}\int\limits_\Om \left| \sqrt{\phi(u^{(k)}_j)}-\sqrt{\phi(u^{(k)}_{j-1})} \right|^2\,dx\right)^{1/2}\nonumber\\
        &\leq \ve_\alpha\sqrt{M|\Om_T|}.
    \end{align}

    Now, for $\alpha\in\NN$ and $k,k'\in\varkappa$ we have
    \begin{align*}
        \norm{\sqrt{\phi(u^{(k)})}-\sqrt{\phi(u^{(k')})}}_{L^1(\Om_T)}&\leq \norm{\sqrt{\phi(u^{(k)})}-\sqrt{\phi(w_{k,\alpha})}}_{L^1(\Om_T)}+\norm{\sqrt{\phi(w_{k,\alpha})}-\sqrt{\phi(w_{k',\alpha})}}_{L^1(\Om_T)}\\
        &\qquad+\norm{\sqrt{\phi(w_{k,\alpha})}-\sqrt{\phi(u^{(k')})}}_{L^1(\Om_T)}\\
        &\stackrel{\cref{continuitybound2}}{\leq} 2\ve_\alpha\sqrt{M|\Om_T|}+  \norm{\sqrt{\phi(w_{k,\alpha})}-\sqrt{\phi(w_{k',\alpha})}}_{L^1(\Om_T)}
    \end{align*}

    Now, applying \cref{l1convergence} to above, we get for any $\alpha\in\NN$
    \begin{align*}
       \limsup_{\varkappa\ni k,k'\to\infty} \norm{\sqrt{\phi(u^{(k)})}-\sqrt{\phi(u^{(k')})}}_{L^1(\Om_T)}{\leq} 2\ve_\alpha\sqrt{M|\Om_T|}.
    \end{align*} This estimate implies that the sequence $\left\{\sqrt{\phi(u^{(k)})}\right\}_{k\in\varkappa}$ is Cauchy in $L^1(\Om_T)$. Let $\Phi\in L^1(\Om_T)$ be its limit. We clain that $\Phi=\sqrt{\phi(u)}$. To this end, for $N\in\NN$, we define the truncation operator 
    \begin{align*}
        T_N(s):=\begin{cases}
            -N&\mbox{ if }s<-N,\\
            s&\mbox{ if }-N\leq s \leq N,\\
            N&\mbox{ if }s>N.
            \end{cases}
    \end{align*}
    From the strong convergence of $\left\{\sqrt{\phi(u^{(k)})}\right\}_{k\in\varkappa}$ in $L^1(\Om_T)$ and the weak convergence of $u^{(k)}$ in $L^p(\Om_T)$, we get for any $N\in\NN$
    \begin{align*}
        0\leq \lim_{\varkappa\ni k\to\infty} \iint\limits_{\Om_T} T_N\left(\sqrt{\phi(v)}-\sqrt{\phi(u^{(k)})}\right)\,\left(v-u^{(k)}\right)\,dx\,dt = \iint\limits_{\Om_T} T_N\left(\sqrt{\phi(v)}-\Phi\right)\,\left(v-u\right)\,dx\,dt
    \end{align*} for any $v\in L^p(\Om_T)$. Now replace $v$ by $u+\delta v$ for $\delta>0$, followed by division by $\delta$ and take limit as $\delta\to 0$ to get for any $N\in\NN$:
    \begin{align*}
        0\leq \lim_{\delta\downarrow 0}\iint\limits_{\Om_T} T_N\left(\sqrt{\phi(u+\delta v)}-\Phi\right)\,\left(v\right)\,dx\,dt = \iint\limits_{\Om_T} T_N\left(\sqrt{\phi(u)}-\Phi\right)\,\left(v\right)\,dx\,dt.
    \end{align*}
    Replacing $v$ by $-v$, we get equality above and since it holds for any $N\in\NN$, we conclude that $\Phi=\sqrt{\phi(u)}$. Then, passing to a further subsequence, we get pointwise a.e. convergence of $u^{(k)}$ to $u$.
\end{proof}

\section{Existence of variational solutions}

The proof is based on the method of minimizing movements in which we discretize time and consider minimizers of time-discretized functionals. This results in the construction of piecewise constant functions as approximating solutions. As the number of steps in the discretization goes to infinity, we recover variational solutions of the original problem. Much of the calculation is analogous to \cite[Section 4]{bogeleinDoublyNonlinearEquations2018}.

\subsection{Approximating functionals}

We fix a step size $h\in (0,1]$. We will inductively construct $u_i\in L^\phi(\Om)\cap W^{s,p}_{u_0}(\Om)$ of non-negative minimizers to certain elliptic variational functionals whenever $i\in\NN$ and $ih\leq T$. Suppose for some $i\in \NN$, $0\leq u_{i-1}\in L^\phi(\Om)\cap W^{s,p}_{u_0}(\Om)$ has been defined. We initialize with $u_{i=0}=u_0$ where $u_0$ is the time-independent initial data. We define the variational functional
\begin{align}\label{approxfunctional}
    F_i[v]:= \iint\limits_{\RR^N\times\RR^N} \frac{H(x,y,u(x,t)-u(y,t))}{|x-y|^N}\,dx\,dy+\frac{1}{h}\int\limits_{\Om} \mathfrak{b}[u_{i-1},v]\,dx,
\end{align} where $\mathfrak{b}[u_{i-1},v]=\phi(v)-\phi(u_{i-1})(v-u_{i-1})$. We choose $u_i$ as the minimizer of $F_i$ in the class of non-negative functions in $L^\phi(\Om)\cap W^{s,p}_{u_0}(\Om)$. This class is non-empty since $u_0$ is an admissible function. The proof of existence of $u_i$ is by the direct method of calculus of variations and relies on the convexity of $H$ as well as the coercivity of the functional.

\begin{proposition}
    \label{existenceapprox}
    Assume that $u_0$ satisfies the hypothesis \cref{datahypo} and $u_*\in L^\phi(\Om)$ with $u_*\geq 0$ a.e. in $\Om$. Then there exists $u\in L^\phi(\Om)\cap W^{s,p}_{u_0}(\Om)$ with $u\geq 0$ a.e. in $\RR^N$ minimizing the functional
    \begin{align*}
        F_i[w]:= \iint\limits_{\RR^N\times\RR^N} \frac{H(x,y,w(x,t)-w(y,t))}{|x-y|^N}\,dx\,dy+\frac{1}{h}\int\limits_{\Om} \mathfrak{b}[u_{*},w]\,dx,
    \end{align*}
    in the class $L^\phi(\Om)\cap W^{s,p}_{u_0}(\Om)$.
\end{proposition}

\begin{proof}
    Let $(u_j)_{j\in\NN}$ be a minimizing sequence in $\mathfrak{X}:=\{v\in L^\phi(\Om)\cap W^{s,p}_{u_0}(\Om):v\geq 0\mbox{ a.e. in }\Om\}$ such that $\lim\limits_{j\to\infty} F[u_j]=\inf_{w\in\mathfrak{X}} F[w]$. By \cref{techlem3}, we receive the following upper bound
    \begin{align*}
        \phi(u_j)&\leq 2\mathfrak{b}[u_*,u_j]+2^{2+1/l}\phi^{*}(b(u_*))\\
        & \leq 2\mathfrak{b}[u_*,u_j]+2^{2+1/l}m\phi(u_*).
    \end{align*}
    Using the lower bound on the functional \cref{coercivity2}, we get 
    \begin{align*}
        \frac{1}{2h}\int_\Om \phi(u_j)+\frac{1}{C_1}||u_j||^p_{W^{s,p}(\RR^N)}&\leq \iint\limits_{\RR^N\times\RR^N} \frac{H(x,y,u_j(x,t)-u_j(y,t))}{|x-y|^N}\,dx\,dy\\ 
        &\qquad+ \frac{C_2}{C_1} ||u_0||_{W^{s,p}(\RR^N)}+\frac{1}{h}\int_\Om \mathfrak{b}[u_*,u_j]+2^{1+1/l}m\phi(u_*)\,dx\\
        &<\infty 
    \end{align*} independent of $j$. Therefore, the sequence $(u_j)_{j\in\NN}$ is bounded in $L^\phi(\Om)\cap W^{s,p}_{u_0}(\Om)$. Hence, for a subsequence, there exists $u\in L^\phi(\Om)\cap W^{s,p}_{u_0}(\Om)$ with $u\geq 0$ a.e. in $\Om$ such that
    \begin{align*}
        u_j\rightharpoonup u\mbox{ in }W^{s,p}_{u_0}(\Om)-\mbox{weak},\\
        u_j\rightharpoonup u\mbox{ in }L^\phi(\Om)-\mbox{weak},
    \end{align*} since weak convergence preserves non-negativity. Recall that $L^\phi(\Om)\cap W^{s,p}_{u_0}(\Om)$ is reflexive due to doubling property of $\phi$ and since $p>1$.

    Due to Fatou's lemma, $F$ is lower semicontinuous with respect to strong convergence in $L^\phi(\Om)\cap W^{s,p}_{u_0}(\Om)$. Since $F$ is convex, it is also lower semicontinuous with respect to weak convergence (see \cref{lowsemcont}).

    As a consequence, we have
    \begin{align*}
        F[u]\leq \liminf\limits_{j\to\infty} F[u_j]=\inf_{w\in\mathfrak{X}}F[w].
    \end{align*} Hence $u$ is a minimizer.
\end{proof}

\subsection{Energy bounds}

The non-negative function $u_{i-1}\in L^\phi(\Om)\cap W^{s,p}_{u_0}(\Om)$ is an admissible function for the functional $F_i$ in \cref{approxfunctional} for any $i\in\NN$. Therefore, by minimality of $u_i$, we get 
\begin{align*}
    \iint\limits_{\RR^N\times\RR^N} \frac{H(x,y,u_i(x,t)-u_i(y,t))}{|x-y|^N}\,dx\,dy&+\frac{1}{h}\int\limits_{\Om} \mathfrak{b}[u_{i-1},u_i]\,dx=F[u_i]\\
    & \leq F[u_{i-1}]=\iint\limits_{\RR^N\times\RR^N} \frac{H(x,y,u_{i-1}(x,t)-u_{i-1}(y,t))}{|x-y|^N}\,dx\,dy.
\end{align*}

By iterating the previous inequality, we obtain 

\begin{align}\label{energyest1}
    \iint\limits_{\RR^N\times\RR^N} \frac{H(x,y,u_k(x,t)-u_k(y,t))}{|x-y|^N}\,dx\,dy&+\sum_{i=1}^k\frac{1}{h}\int\limits_{\Om} \mathfrak{b}[u_{i-1},u_i]\,dx\\
    & \leq \iint\limits_{\RR^N\times\RR^N} \frac{H(x,y,u_0(x,t)-u_0(y,t))}{|x-y|^N}\,dx\,dy,
\end{align} whenever $h>0$, $k\in\NN$ with $kh\leq T$.

By \cref{techlem4} and \cref{energyest1} and the non-negativity of $H$, we obtain 

\begin{align}\label{energyest2}
    \sum_{i=1}^k \int\limits_{\Om} \left| \sqrt{\phi(u_i)}-\sqrt{\phi(u_{i-1})} \right|^2\,dx \leq C \sum_{i=1}^k\int\limits_{\Om} \mathfrak{b}[u_{i-1},u_i]\,dx\leq ChM,
\end{align} where \[M=\iint\limits_{\RR^N\times\RR^N} \frac{H(x,y,u_0(x,t)-u_0(y,t))}{|x-y|^N}\,dx\,dy.\]

As a result,

\begin{align}
    \label{energyest3}
    \int\limits_{\Om} \phi(u_k)\,dx &\leq 2\int\limits_{\Om} \left|\sqrt{\phi(u_k)}-\sqrt{\phi(u_0)}\right|^2\,dx + 2\int\limits_{\Om} \phi(u_0)\,dx\nonumber \\
    & \leq 2k \sum_{i=1}^k\int\limits_{\Om} \left|\sqrt{\phi(u_i)}-\sqrt{\phi(u_{i-1})}\right|^2\,dx + 2\int\limits_{\Om} \phi(u_0)\,dx \nonumber\\
    & \leq CMT + 2\int\limits_{\Om} \phi(u_0)\,dx,
\end{align}

and also 

\begin{align}
    \label{energyest4}
    \iint\limits_{\RR^N\times\RR^N} \frac{|u_k(x)-u_k(y)|^p}{|x-y|^{N+ps}}\,dx\,dy & \leq C_1 \iint\limits_{\RR^N\times\RR^N} \frac{H(x,y,u_k(x)-u_k(y))}{|x-y|^N}\,dx\,dy+C_2 ||u_0||^p_{W^{s,p}(\RR^N)}\nonumber\\
    &\leq C_1 \iint\limits_{\RR^N\times\RR^N} \frac{H(x,y,u_0(x)-u_0(y))}{|x-y|^N}\,dx\,dy+C_2 ||u_0||^p_{W^{s,p}(\RR^N)}\nonumber\\ 
    &\leq C_1M +C_2 ||u_0||^p_{W^{s,p}(\RR^N)}.
\end{align}

\subsection{Convergence to a limit function}

In the sequel, we will consider such $h=h(k)\in (0,1]$ such that $h_k=\frac{T}{k}$ with $k\in\NN$. We define the map 
\begin{align*}
    u^{(k)}(\cdot,t):=u_i\mbox{ for }t\in ((i-1)h_k,ih_k)\mbox{ with }i\in\{0,1,\ldots,k\}
\end{align*} where $u_i=u_0$ for $i=0$.

From \cref{energyest3} and \cref{energyest4}, we get 

\begin{align}\label{energyest5}
    \sup\limits_{t\in [0,T]} \int\limits_\Om \phi(u^{(k)}(t))\,dx &+ \sup_{t\in [0,T]} \iint\limits_{\RR^N\times\RR^N} \frac{H(x,y,u^{(k)}(x)-u^{(k)}(y))}{|x-y|^N}\,dx\,dy\\
    &\leq C_1 M + 2\int_\Om\phi(u_0)\,dx + C_2 ||u_0||_{W^{s,p}(\RR^N)}^p.
\end{align}

The estimate in \cref{energyest5} guarantees boundedness in the space $L^\infty(0,T;W^{s,p}_{u_0}(\Om))$. Therefore, there exists a subsequence of $(u^{(k)})$ and a limit function $u\in L^\infty(0,T;W^{s,p}_{u_0}(\Om))$ such that as $k\to\infty$ for the not-relabelled sequence, we have 
\begin{align}\label{weakcon1}
    u^{(k)}\rightharpoonup u \mbox{ in } L^\infty(0,T;W^{s,p}_{u_0}(\Om))-\mbox{weak}-*.
\end{align}

By \cref{energyest5}, \cref{energyest2} and \cref{weakcon1}, the hypothesis of \cref{compactnesslemma} are satisfied. As a result, for a further not-relabelled subsequence, we have 

\begin{align}\label{convergencept}
    &\sqrt{\phi(u^{(k)})}\to \sqrt{\phi(u)} \mbox{ in }L^1(\Om_T)-\mbox{strong}\\
    &u^{(k)}\to u\mbox{ a.e in }\RR^N\times(0,T).
\end{align}

Now, observe that by \cref{energyest2}

\begin{align*}
    \int_0^T\int\limits_{\Om} |\Delta_{h_k} \sqrt{\phi(u^{(k)})}|^2\,dx\,dt &= \int_0^T\int\limits_{\Om} \left|\frac{1}{h_k}\left(\sqrt{\phi(u^{(k)}(t+h_k))}-\sqrt{\phi(u^{(k)}(t))}\right)\right|^2\,dx\,dt\\
    & = \int\limits_\Om h_k \sum_{i=1}^k \left| \frac{1}{h_k}\left( \sqrt{\phi(u_i)}-\sqrt{\phi(u_{i-1})} \right) \right|^2\,dx\\
    & =  \frac{1}{h_k}\int\limits_\Om\sum_{i=1}^k \left| \sqrt{\phi(u_i)}-\sqrt{\phi(u_{i-1})} \right|^2\,dx\leq C.
\end{align*}

Therefore, for a further, not-relabelled subsequence, there exists $w\in L^2(\Om_T)$ such that 

\begin{align}\label{weakcon2}
    \Delta_{h_k} \sqrt{\phi(u^{(k)})}\rightharpoonup w\mbox{ in }L^2(\Om_T)-\mbox{weak}.
\end{align}

Now we calculate for any $\varphi\in C_0^\infty(\Om_T)$

\begin{align*}
    \iint\limits_{\Om_T} \sqrt{\phi(u)} \de_t\varphi\,dx\,dt &= \lim\limits_{k\to\infty} \iint\limits_{\Om_T} \sqrt{\phi(u^{(k)})} \Delta_{-h_k}\varphi\,dx\,dt \\
    & = -\lim\limits_{k\to\infty} \iint\limits_{\Om_T} \Delta_{h_k} \sqrt{\phi(u^{(k)})} \, \varphi\,dx\,dt\\
    & = - \iint\limits_{\Om_T} w \, \varphi\,dx\,dt,
\end{align*} where in the first equality we used pointwise a.e. convergence of $\sqrt{\phi(u^{(k)})}$ and strong convergence of $\Delta_{-h_k}\varphi$, in the second inequality, we used discrete integration by parts and in the last equality, we used \cref{weakcon2}.

The above calculation implies that $w=\de_t\sqrt{\phi(u)}$. Therefore, $\de_t\sqrt{\phi(u)}\in L^2(\Om_T)$. As a result, $\de_t{\phi(u)}=2\sqrt{\phi(u)}\de_t\sqrt{\phi(u)}\in L^1(\Om_T)$. This implies that $\phi(u)\in C^0([0,T];L^1(\Om))$, which, by \cref{convergence2} implies that $u\in C^0([0,T];L^\phi(\Om))$.

\subsection{A variational inequality for approximate minimizers}

Let us observe that the map $u^{(k)}$ minimizes the integral functional defined by 
\begin{align}
    \label{approxfunctional2}
    F^{(k)}[v]:=\int_{0}^T\iint\limits_{\RR^N\times\RR^N} \frac{H(x,y,v(x,t)-v(y,t))}{|x-y|^N}\,dx\,dy\,dt + \frac{1}{h_k} \int_0^T\int\limits_{\Om} \mathfrak{b}[u^{(k)}(t-h_k),v(t)]\,dx\,dt
\end{align} in the class of non-negative functions $v\in L^\phi(\Om_T)\cap L^p(0,T;W^{s,p}_{u_0}(\RR^N))$.

\paragraph{Justification} We have the following chain of inequalities 
\begin{align*}
    F^{(k)}[u^{(k)}]&:=\int_{0}^T\iint\limits_{\RR^N\times\RR^N} \frac{H(x,y,u^{(k)}(x,t)-u^{(k)}(y,t))}{|x-y|^N}\,dx\,dy\,dt + \frac{1}{h_k} \int_0^T\int\limits_{\Om} \mathfrak{b}[u^{(k)}(t-h_k),u^{(k)}(t)]\,dx\,dt\\
    &=\sum_{i=1}^k\iiint\limits_{(\RR^N\times\RR^N)\times((i-1)h_k,ih_k)}\hspace{-1cm} \frac{H(x,y,u_i(x)-u_{i-1}(y))}{|x-y|^N}\,dx\,dy\,dt + \sum_{i=1}^k\frac{1}{h_k} \iint\limits_{\Om\times((i-1)h_k,ih_k)} \mathfrak{b}[u_{i-1},u_i]\,dx\,dt\\
    &=\sum_{i=1}^k \int_{(i-1)h_k}^{ih_k} F_i[u_i]\,dt \leq \sum_{i=1}^k\int_{(i-1)h_k}^{ih_k} F_i[v]\,dt = F^{(k)}[v],
\end{align*}
 where the minimizing property of $u_i$ on each time interval is used (\cref{existenceapprox}).

Now, for a fixed comparision function $v\in L^\phi(\Om_T)\cap L^p(0,T;W^{s,p}_{u_0}(\Om))$, observe that for any $s\in (0,1)$, the convex combination $w_s:=u^{(k)}+s(v-u^{(k)})$ is also admissible. Therefore, by minimality of $u^{(k)}$, we have
\begin{align*}
    F^{(k)}[u^{(k)}]\leq F^{(k)}[w_s]&\leq \int_0^T \iint\limits_{\RR^N\times\RR^N} (1-s) \frac{H(x,y,u^{(k)}(x)-u^{(k)}(y))}{|x-y|^N}\,dx\,dy \\
    & \qquad + \int_0^T \iint\limits_{\RR^N\times\RR^N} s \frac{H(x,y,v(x,t)-v(y,t))}{|x-y|^N}\,dx\,dy\\
    &\qquad\qquad+\frac{1}{h_k}\int_0^T\int\limits_{\Om} \mathfrak{b}[u^{(k)}(t-h_k),w_s]\,dx\,dt. 
\end{align*}

This implies
\begin{align*}
    \int_0^T \iint\limits_{\RR^N\times\RR^N} &\frac{H(x,y,u^{(k)}(x)-u^{(k)}(y))}{|x-y|^N}\,dx\,dy \leq \int_0^T \iint\limits_{\RR^N\times\RR^N} \frac{H(x,y,v(x,t)-v(y,t))}{|x-y|^N}\,dx\,dy\\
    & +\frac{1}{sh_k}\int_0^T\int\limits_{\Om} \mathfrak{b}[u^{(k)}(t-h_k),w_s]\,dx\,dt-\frac{1}{sh_k}\int_0^T\int\limits_{\Om} \mathfrak{b}[u^{(k)}(t-h_k),u^{(k)}]\,dx\,dt\\
    &= \int_0^T \iint\limits_{\RR^N\times\RR^N} \frac{H(x,y,v(x,t)-v(y,t))}{|x-y|^N}\,dx\,dy +\frac{1}{h_k}\int_0^T\int\limits_{\Om} \frac{1}{s} \left(\phi(w_s)-\phi(u^{(k)})\right)\\
    &\qquad\qquad-\frac{1}{h_k}\int_0^T\int\limits_{\Om} b(u^{(k)}(t-h_k))(v-u^{(k)})\,dx\,dt.
\end{align*}

The function $s\mapsto \frac{1}{s}\left(\phi(w_s)-\phi(u^{(k)})\right)$ is monotone and converges a.e. to $b(u^{(k)})(v-u^{(k)})$. Hence, passing to the limit as $s\to 0$ by dominated convergence theorem, we obtain 
\begin{align}\label{varineqapprox}
    \int_0^T \iint\limits_{\RR^N\times\RR^N} \frac{H(x,y,u^{(k)}(x)-u^{(k)}(y))}{|x-y|^N}&\,dx\,dy \leq  \int_0^T \iint\limits_{\RR^N\times\RR^N} \frac{H(x,y,v(x)-v(y))}{|x-y|^N}\,dx\,dy\nonumber\\
    &\qquad\qquad+\int_0^T\int\limits_{\Om} \Delta_{-h_k}b(u^{(k)})(v-u^{(k)})\,dx\,dt.
\end{align}

The preceding inequality \cref{varineqapprox} can be localized to the time interval $(0,\tau)$ by choosing the comparision function $v=\chi_{(0,\tau)}v+\chi_{(\tau,T)}u^{(k)}$. This gives us 
\begin{align}\label{varineqapprox2}
    \int_0^T \iint\limits_{\RR^N\times\RR^N} \frac{H(x,y,u^{(k)}(x)-u^{(k)}(y))}{|x-y|^N}&\,dx\,dy \leq  \int_0^T \iint\limits_{\RR^N\times\RR^N} \frac{H(x,y,v(x,t)-v(y,t))}{|x-y|^N}\,dx\,dy\nonumber\\
    &\qquad\qquad+\int_0^T\int\limits_{\Om} \Delta_{-h_k}b(u^{(k)})(v-u^{(k)})\,dx\,dt.
\end{align}

\subsection{Passage to limit in variational inequality}

To the inequality \cref{varineqapprox2}, we apply the discrete integration by parts lemma as in \cref{discreteintbyparts} to get 
\begin{align}\label{varineq2}
    \int_0^\tau \iint\limits_{\RR^N\times\RR^N} &\frac{H(x,y,u^{(k)}(x)-u^{(k)}(y))}{|x-y|^N}\,dx\,dy \leq  \int_0^\tau \iint\limits_{\RR^N\times\RR^N} \frac{H(x,y,v(x,t)-v(y,t))}{|x-y|^N}\,dx\,dy\nonumber\\
    &+\underbrace{\int_0^\tau\int\limits_{\Om} \Delta_{h_k}v\,(b(v)-b(u^{(k)}))\,dx\,dt}_{A}-B_\tau(h_k)+B_0(h_k)+\delta_1(h_k)+\delta_2(h_k)
\end{align} for every $\tau\in (0,T]$. 

The latter four terms are
\begin{align*}
    B_\tau(h_k)&:=\frac{1}{h_k}\iint\limits_{\Om\times (\tau-h_k,\tau)} \mathfrak{b}[u^{(k)}(t),v(t+h_k)]\,dx\,dt,\\
    B_0(h_k)&:=\frac{1}{h_k}\iint\limits_{\Om\times (-h_k,0)}\mathfrak{b}[u^{(k)},v]\,dx\,dt,\\
    \delta_1(h_k)&:=\frac{1}{h_k} \iint\limits_{\Om_\tau} \mathfrak{b}[v(t),v(t+h_k)]\,dx\,dt,\\
    \delta_2(h_k)&:=\iint\limits_{\Om\times(-h_k,0)}\Delta_{h_k}v\,\left(b(v(t+h_k))-b(u^{(k)})\right)\,dx\,dt.
\end{align*}

By \cref{discreteintbyparts}, if $\de_t v\in L^\phi(\Om_T)$, then 
\begin{align}\label{limbound2}
    &\lim_{k\to\infty} \delta_1(h_k)=0,\mbox{ and }\\
    &\lim_{k\to\infty} \delta_2(h_k)=\lim_{k\to\infty}\iint\limits_{\Om\times(-h_k,0)}\Delta_{h_k}v\,\left(b(v(t+h_k))-b(u_0)\right)\,dx\,dt=0.
\end{align}

The comparision map $v$ in \cref{varineq2} is chosen to satisfy $v\in L^p(0,T;W^{s,p}_{u_0}(\Om))$ with $\de_t v\in L^\phi(\Om_T)$ and $v(0)\in L^\phi(\Om)$. Moreover, we extend $v$ to negative times $t<0$ by $v(t)=v(0)$.

Let us look at the term $A$. Notice that $\Delta_{h_k}v\to \de_t v$ in $L^\phi(\Om_T)$-strong. This can be proved in two steps. The mollification of $v$ is strongly convergent by \cite[Theorem 8.21]{adamsSobolevSpaces2003} and the difference quotients converge strongly to derivative for smooth functions.

On the other hand, $b(u^{(k)})$ is bounded uniformly in $L^{\phi^*}(\Om_T)$ by \cref{techlem2} and \cref{energyest5}. The space $L^{\phi^*}(\Om_T)$ is reflexive due to the doubling property of $\phi^*$ (\cite[Theorem 8.20]{adamsSobolevSpaces2003}). Therefore, for a subsequence $b(u^{(k)})$ converges weakly to some limit in $L^{\phi^*}(\Om_T)$. However, due to \cref{convergencept}, the limit can be readily identified as $b(u)$ by continuity of $b$.

As a result, we can pass to the limit:
\begin{align}\label{covergence5}
    \lim_{k\to\infty} \iint\limits_{\Om_{\tau}} \Delta_{h_k}v\, (b(v)-b(u^{(k)}))\,dx\,dt=\iint\limits_{\Om_{\tau}} \de_t v\, (b(v)-b(u))\,dx\,dt.
\end{align}

Now, let us look at $B_0(h_k)$. Since $u^{(k)}$ and $v$ are constant over the interval $(-h_k,0)$, we have
\begin{align}\label{B0}
    B_0(h_k)=\int_\Om \mathfrak{b}[u_0,v(0)]\,dx.
\end{align}

Except for the expression $B_\tau(h_k)$, it is possible to pass to the limit as $k\to\infty$ in all the expressions in \cref{varineq2}. To handle $B_{\tau}(h_k)$, we integrate over $\tau\in(t_0,t_0+\delta)$ for $t_0\in [0,T-\delta]$ and $\delta\in (0,T)$ and divide by $\delta$. As a consequence, \cref{varineq2} becomes 

\begin{align}\label{varineq3}
    \int_0^{t_0} \iint\limits_{\RR^N\times\RR^N} &\frac{H(x,y,u^{(k)}(x)-u^{(k)}(y))}{|x-y|^N}\,dx\,dy \leq  \int_0^{t_0+\delta} \iint\limits_{\RR^N\times\RR^N} \frac{H(x,y,v(x)-v(y))}{|x-y|^N}\,dx\,dy\nonumber\\
    &+\fint_{t_0}^{t_0+\delta}\iint\limits_{\Om_{\tau}} \Delta_{h_k}v\,(b(v)-b(u^{(k)}))\,dx\,dt-\frac{1}{\delta}\iint\limits_{\Om\times (t_0,t_0+\delta-h_k)} \mathfrak{b}[u^{(k)}(t),v(t+h_k)]\,dx\,dt\nonumber\\
    &+\int_\Om \mathfrak{b}[u_0,v(0)]\,dx+\delta_1(h_k)+\delta_2(h_k)
\end{align}

Now, $u^{(k)}\to u$ pointwise a.e. by \cref{convergencept} and since $v\in C^0([0,T];L^1(\Om))$, we have $v(t+h_k)\to v(t)$ a.e. in $\Om_T$ as $k\to\infty$. By non-negativity of $\mathfrak{b}$, we can apply Fatou's lemma to conclude 

\begin{align}\label{limbound1}
    \frac{1}{\delta} \int_{t_0}^{t_0+\delta} \int\limits_{\Om} \mathfrak{b}[u(t),v(t)]\,dx\,dt \leq \liminf_{k\to\infty} \frac{1}{\delta}\iint\limits_{\Om\times (t_0,t_0+\delta-h_k)} \mathfrak{b}[u^{(k)}(t),v(t+h_k)]\,dx\,dt.
\end{align}

Using weak lower semicontinuity of the integral $\int_0^T\iint_{\RR^N\times\RR^N} \frac{H(x,y,u(x,t)-u(y,t))}{|x-y|^N}$ or by Fatou's lemma and the convergences \cref{convergencept}, \cref{limbound1}, \cref{limbound2}, \cref{covergence5}, we get 

\begin{align}\label{varineq4}
    \int_0^{t_0} \iint\limits_{\RR^N\times\RR^N} &\frac{H(x,y,u(x,t)-u(y,t))}{|x-y|^N}\,dx\,dy \leq  \int_0^{t_0+\delta} \iint\limits_{\RR^N\times\RR^N} \frac{H(x,y,v(x,t)-v(y,t))}{|x-y|^N}\,dx\,dy\nonumber\\
    &+\fint_{t_0}^{t_0+\delta}\iint\limits_{\Om_{\tau}} \de_t v\,(b(v)-b(u))\,dx\,dt- \frac{1}{\delta} \int_{t_0}^{t_0+\delta} \int\limits_{\Om} \mathfrak{b}[u(t),v(t)]\,dx\,dt \nonumber\\
    &+\int_\Om \mathfrak{b}[u_0,v(0)]\,dx.
\end{align}

It remains to pass to the limit in $\delta\to 0$. The first and the second term on the right hand side depends continuously on $\delta$ due to absolute continuity on integration in time. For the third term, we have
\begin{align}\label{bound3}
    0\leq \mathfrak{b}[u(t),v(t)]\stackrel{\cref{boundary1}}{\leq} \phi(v(t))+\phi^*(b(u(t)))&\stackrel{\cref{techlem2}}{\leq} \phi(v(t))+\frac{m}{m+1}u(t)b(u(t))\nonumber\\
    &\stackrel{\cref{techlem2}}{\leq} \phi(v(t))+m\phi(u(t)).
\end{align}

Since $u,v\in C^0([0,T];L^\phi(\Om))$, the right hand side of \cref{bound3} is continuous in time. Therefore, by an application of dominated convergence theorem (\cite[Theorem 1.20]{evansMeasureTheoryFine2015}), the function \[[0,T]\ni t\mapsto \int\limits_\Om \mathfrak{b}[u(t),v(t)]\,dx\] is continuous. 

Hence, now, we can pass to the limit as $\delta\to 0$ in \cref{varineq4} to get 
\begin{align*}
    \int_0^{t_0} \iint\limits_{\RR^N\times\RR^N} &\frac{H(x,y,u(x,t)-u(y,t))}{|x-y|^N}\,dx\,dy \leq  \int_0^{t_0} \iint\limits_{\RR^N\times\RR^N} \frac{H(x,y,v(x,t)-v(y,t))}{|x-y|^N}\,dx\,dy\nonumber\\
    &+\iint\limits_{\Om_{\tau}} \de_t v\,(b(v)-b(u))\,dx\,dt- \int\limits_{\Om} \mathfrak{b}[u(t_0),v(t_0)]\,dx\,dt +\int_\Om \mathfrak{b}[u_0,v(0)]\,dx,
\end{align*}
for all $v\in L^p(0,T;W^{s,p}_{u_0}(\Om))$ with $\de_t v\in L^\phi(\Om_T)$ and $v(0)\in L^\phi(\Om)$. This completes the proof of \cref{mainthm:e}.

\section{Variational inequality for weak solutions}

In this section, we compare the notion of weak solutions to that of variational solutions. We will first show that if the function $\xi\mapsto H(x,y,\xi)$ is $C^1$ and satisfies a comparable growth condition from above namely
\begin{align}
    \label{eq:bound_above_H} H(x,y,\xi) \leq C\left(\frac{|\xi|}{|x-y|^s}\right)^p\mbox{ and } |D_\xi H(x,y,\xi)| \leq C\frac{|\xi|^{p-1}}{|x-y|^{sp}}
\end{align}
then a variational solution is a weak solution in the sense given in the following theorem, which is analogous to \cite[Theorem 6.1]{bogeleinDoublyNonlinearEquations2018}.

\begin{theorem}
    Let $\Om$ be an open bounded subset of $\RR^N$. Suppose that $H$ satisfies the assumptions~\cref{eq:bound_below_H} ~\cref{eq: cvx_H} and ~\cref{eq:bound_above_H} and let the time-independent Cauchy-Dirichlet data $u_0:\RR^N\to[0,\infty)$ satisfy~\cref{datahypo}. Then the solution  \[u\in L^p(0,T;W^{s,p}(\mathbb{R}^N))\cap C^0(0,T;L^\phi(\Om)), \mbox{ such that } u-g\in L^p(0,T;W^{s,p}_0(\Om))\] as in \cref{mainthm:e} satisfies the following variational inequality for all $\tau\in [0,T]$:
     \begin{align}
         \label{defweak}
         \int_{0}^{\tau}\iint\limits_{\mathbb{R}^N\times\mathbb{R}^N} &\frac{D_{\xi}H(x,y,u(x,t)-u(y,t))((v-u)(x,t)-(v-u)(y,t))}{|x-y|^N}\,dx\,dy\,dt+\nonumber\\
         &\int_{0}^{\tau}\int_{\Om} \de_t v(b(v)-b(u))\,dx\,dt \geq \mathfrak{B}[u(\tau),v(\tau)]-\mathfrak{B}[u_0,v(0)],
        \end{align} for all $v:(0,T)\times\RR^N\to [0,\infty)$ such that $v\in L^p(0,T;W^{s,p}(\mathbb{R}^N))$ and $\de_t v\in L^\phi(0,T;\Om)$ such that $v-u_0\in L^p(0,T;W_0^{s,p}(\Om))$ and $v(0)\in L^\phi(\Om)$.
\end{theorem}

\begin{proof}
    For a comparison function $v$ as stated in the theorem, let us test inequality \eqref{defvar} with the comparison maps
    $$w_h = [u]_h + s(v - [u]_h),$$ where $s \in (0,1), \; h> 0$ and $[u]_h$ denotes the time mollification with initial values $v_0 = u_0.$  This implies that 
    \begin{align}\label{27DEC}
        \mathfrak{B}[u(\tau), w_{h}(\tau)] + \int_{0}^{\tau}\iint\limits_{\mathbb{R}^N\times\mathbb{R}^N} &\frac{H(x,y,u(x,t)-u(y,t)) - H(x,y,w_{h}(x,t)-w_{h}(y,t))}{|x-y|^N}\,dx\,dy\,dt\nonumber\\ &\leq  \mathfrak{B}[u_0, w_{h}(0)] + \int_{0}^{\tau}\int_{\Om} \de_t w_h(b(w_h)-b(u))\,dx\,dt
    \end{align}
    Let us re-write the integral involving the time derivative in the following form:
    \begin{align*}
        \int_{0}^{\tau}\int_{\Om} \de_t w_h(b(w_h)-b(u))\,dx\,dt &=  \int_{0}^{\tau}\int_{\Om}\big[\de_t w_h b(w_h)- (1 - s)\de_t [u]_h b(u) - s\de_t v b(u) \big] \,dx\,dt\\
        &=  \int_{0}^{\tau}\int_{\Om}\big[\de_t(\phi(w_h))- (1 - s)\de_t \phi([u]_h)   - s \de_t v b(u) \big] \,dx\,dt \\ &+ (1 - s)  \int_{0}^{\tau}\int_{\Om} \de_t [u]_h \big(b([u]_h)-b(u)\big)\,dx\,dt.
        \end{align*}
        By the monotonicity of $b$ and the identity in \cref{generalmolliX}, the last integral is non-positive. Moreover, let us use  the identity
        \begin{align*}
            s \int_{0}^{\tau}\int_{\Om} \de_t v(b(v))\,dx\,dt &= s \int_{0}^{\tau}\int_{\Om} \de_t (\phi(v))\,dx\,dt \\
            & =  s \int_{\Om \times \tau} \phi(v)\,dx - s \int_{\Om \times 0} \phi(v)\,dx
        \end{align*}
        and corresponding identities for $w_h$ and $[u]_h,$ to get the following estimates:
        \begin{align*}
        \int_{0}^{\tau}\int_{\Om} \de_t w_h(b(w_h)-b(u))\,dx\,dt & \leq s \int_{0}^{\tau}\int_{\Om} \de_t v(b(v)-b(u))\,dx\,dt \\
          & + \int_{\Om \times \{\tau\}} \Big[\phi(w_h) - (1 - s) \phi([u]_h) - s\phi(v)\Big] \,dx\\
          & - \int_{\Om \times \{0\}} \Big[\phi(w_h) - (1 - s) \phi([u]_h) - s\phi(v)\Big] \,dx
        \end{align*}
        Since $u \in C^0([0,T];L^\phi(\Om))$ by construction of $u$ and $L^{\phi}(\Om) \subset L^1(\Om)$, we can apply \cref{propertiesofmolli} (iv).   This implies that 
        \begin{align*}
            [u]_h \longrightarrow u \;\;\mathrm{ a.e. }\; \Om \times \{\tau\},\\
             [u]_h \longrightarrow u \;\;\mathrm{ a.e. }\; \Om \times \{0\}.
        \end{align*}
       Therefore, using convexity of $\phi$ along with Jensen's inequality and \cref{propertiesofmolli} (iv), we can invoke the dominated convergence theorem to get the following estimates: 
        \begin{align}\label{bbest1}
         \limsup_{h \downarrow 0}\; \int_{0}^{\tau}\int_{\Om} \de_t w_h(b(w_h)-b(u))\,dx\,dt & \leq s \int_{0}^{\tau}\int_{\Om} \de_t v(b(v)-b(u))\,dx\,dt \nonumber\\
          & + \int_{\Om \times \{\tau\}} \Big[\phi(u + s(v - u)) - \phi(u) + s(\phi(u) - \phi(v))\Big] \,dx\nonumber\\
          & - \int_{\Om \times \{0\}} \Big[\phi(u + s(v - u)) - \phi(u) + s(\phi(u) - \phi(v))\Big] \,dx.
        \end{align}
      Using the same argument, we can show that 
      \begin{align}\label{bbest2}
          \lim_{h \downarrow 0}\;\Big[ \mathfrak{B}[u_{0}, w_{h}(0)] -   \mathfrak{B}[u(\tau), w_{h}(\tau)] \Big] = \int_{\Om \times \{0\}} \mathfrak{b}\big[u, u + s(v - u)\big]dx - \int_{\Om \times \{\tau\}} \mathfrak{b}\big[u, u + s(v - u)\big]dx.
      \end{align}
    For the term involving the nonlocal operator, we have pointwise a.e. convergence for $w_h$ using \cref{propertiesofmolli} (i). Then, by the $p$-growth condition \cref{eq:bound_above_H} which allows us to use dominated convergence theorem, we get the following estimates:
    \begin{align}\label{bbest3}
           \lim_{h \downarrow 0}&\int_{0}^{\tau}\iint\limits_{\mathbb{R}^N\times\mathbb{R}^N} \frac{H(x,y,w_{h}(x,t)-w_{h}(y,t))}{|x-y|^N}\,dx\,dy\,dt\nonumber\\
           & \longrightarrow     \;\int_{0}^{\tau}\iint\limits_{\mathbb{R}^N\times\mathbb{R}^N} \frac{H(x,y,[u + s(v - u)](x,t)-[u + s(v - u)](y,t))}{|x-y|^N}\,dx\,dy\,dt.
    \end{align}
    Taking limit $h \longrightarrow 0$ in \cref{27DEC} by way of \cref{bbest1}, \cref{bbest2} and \cref{bbest3}, we obtain
    \begin{align*}
       \int_{\Om \times \{\tau\}} &\mathfrak{b}\big[u, u + s(v - u)\big]dx \\
       &+ \;\int_{0}^{\tau}\iint\limits_{\mathbb{R}^N\times\mathbb{R}^N} \frac{H(x,y,u(x,t)-u(y,t)) - H(x,y,[u + s(v - u)](x,t)-[u + s(v - u)](y,t))}{|x-y|^N}\,dx\,dy\,dt\\
       & \leq  \int_{\Om \times \{0\}} \mathfrak{b}\big[u, u + s(v - u)\big]dx + \int_{0}^{\tau}\int_{\Om} \de_t v(b(v)-b(u))\,dx\,dt \\ 
      & \hspace{5mm}+ \int_{\Om \times \{\tau\}} \Big[\phi(u + s(v - u)) - \phi(u) + s(\phi(u) - \phi(v))\Big] \,dx\\
     &\hspace{5mm} - \int_{\Om \times \{0\}} \Big[\phi(u + s(v - u)) - \phi(u) + s(\phi(u) - \phi(v))\Big] \,dx.
    \end{align*}
    We divide the inequality by $s \in (0,1)$ and using the definition of $\mathfrak{b}$, we have
    \begin{align*}
      0 \leq \frac{1}{s}   &\;\int_{0}^{\tau}\iint\limits_{\mathbb{R}^N\times\mathbb{R}^N} \frac{ H(x,y,[u + s(v - u)](x,t)-[u + s(v - u)](y,t)) - H(x,y,u(x,t)-u(y,t)) }{|x-y|^N}\,dx\,dy\,dt\\ 
      & + \int_{0}^{\tau}\int_{\Om} \de_t v(b(v)-b(u))\,dx\,dt +  \mathfrak{B}[u_{0}, v(0)] -   \mathfrak{B}[u(\tau), v(\tau)].
    \end{align*}
    Making use of the growth assumptions on the derivatives of the $H$ and dominated convergence theorem, we can pass to the limit $s \longrightarrow 0$ in the first integral. Finally, we get the following estimates: 
    \begin{align*}
        \int_{0}^{\tau}\iint\limits_{\mathbb{R}^N\times\mathbb{R}^N} &\frac{D_{\xi}H(x,y,u(x,t)-u(y,t))((v-u)(x,t)-(v-u)(y,t))}{|x-y|^N}\,dx\,dy\,dt+\nonumber\\
            &\int_{0}^{\tau}\int_{\Om} \de_t v(b(v)-b(u))\,dx\,dt \geq \mathfrak{B}[u(\tau) ,v(\tau)]-\mathfrak{B}[u_0,v(0)],
    \end{align*} 
    for all $v:(0,T)\times\RR^N\to [0,\infty)$ such that $v\in L^p(0,T;W^{s,p}(\mathbb{R}^N))$ and $\de_t v\in L^\phi(0,T;\Om)$ such that $v-u_0\in L^p(0,T;W_0^{s,p}(\Om))$ and $v(0)\in L^\phi(\Om)$.
    \end{proof}

\section{Variational solutions are distributional solutions}
In this section, we want to check whether the variational solutions are also distributional solutions. We will show that this is the case when the nonlinearity $b$ satisfies
\begin{align}\label{new1}
l\le \frac{u b'(u)}{b(u)}\le m, \mbox{ for given constants } m\ge l\ge 1,
\end{align}
and for functions H satisfying, additionally, 
\begin{align}\label{new2}
H\ge 0 \mbox{ and } H (x, y, 0) = 0 \mbox{ a.e. } x,y \in \mathbb{R}^N\times \mathbb{R}^N
\end{align}

We make a few comments about these assumptions.

The non-negativity of solutions can also be seen as solving an obstacle problem with $0$ function as an obstacle. The assumption \cref{new2} corresponds to the $0$ function being a minimum of the nonlocal functional with zero boundary values.

The assumption $l\ge 1$ is required so that $\de_t b(u)$ does not become singular, that is, it is zero whenever u is zero. The assumption $l\ge 1$ also guarantees Lipschitz continuity of $b$ which matches the assumption in the existence result of Alt \& Luckhaus \cite[Theorem 2.3]{altQuasilinearEllipticparabolicDifferential1983}. However, we can relax their assumption of $p$-growth from above. The following theorem is analogous to \cite[Theorem 7.1]{bogeleinDoublyNonlinearEquations2018}.

\begin{theorem} Assume that the function $H$ satisfies \cref{eq:bound_below_H}, \cref{eq: cvx_H} and \cref{new2} and $\xi\rightarrow H(x,y,\xi)$ is $C^1$. For the function $b$, assume that \cref{new1} holds and the initial and boundary data satisfies \cref{datahypo}. Let u $\in$ $C^0([0,T];L^\phi(\Omega))\cap L^p(0,T;W^{s,p}_{u_0}(\Omega))$ be as obtained in \cref{mainthm:e}. If for every $\psi \in C_0^\infty(\Omega_T)$ there is a function $F \in L^{\frac{p}{p-1}}(\RR^N \times \RR^N \times (0,T))$, which may depend on $u$ and $\psi$, so that
\begin{align}\label{upper1}
\frac{|D_\xi H(x,y,u(x,t)-u(y,t)+s(\psi(x,t)-\psi(y,t)))|}{|x-y|^{N-\frac{N}{p}-s}} \leq F
\end{align}
holds almost everywhere for every $0<s<1$, then u is a solution in the sense of distributions, that is
\begin{align}
\iint\limits_{\Omega_T} -b(u)\de_t\psi \,dx\,dt + \int_0^T\iint\limits_\mathbb{R^N\times R^N} \frac{D_\xi H(x,y,u(x,t)-u(y,t)) (\psi(x,t)-\psi(y,t))}{|x-y|^N} \,dx\,dy\,dt = 0
\end{align}
\end{theorem}
\begin{remark}
    In particular, the assumption \cref{upper1} is satisfied when $H$ satisfies the $p$-growth condition as in \cref{eq:bound_above_H}.
\end{remark}
\begin{proof}
Notice that for every $s\in (0,1)$ and every testing function $\psi \in C^\infty_0(\Omega_T)$, assumption \cref{upper1} implies that
\begin{align}\label{new3}
&\frac{|H(x,y,u(x,t)-u(y,t)+s(\psi(x,t)-\psi(y,t))|}{|x-y|^N}\leq \frac{|H(x,y,u(x,t)-u(y,t)|}{|x-y|^N}\nonumber\\
&\qquad+s\int_0^1 \frac{|D_\xi H(x,y,u(x,t)-u(y,t)+s\sigma (\psi(x,t)-\psi(y,t)))|\psi(x,t)-\psi(y,t)|}{|x-y|^N}d\sigma\nonumber\\
&\qquad\qquad\leq \frac{|H(x,y,u(x,t)-u(y,t)|}{|x-y|^N}+\left(\frac{|\psi(x,t)-\psi(y,t)|}{|x-y|^{\frac{N}{p}+s}}\right) F(x,y,t)
\end{align}
holds a.e. $(x,y,t)\in \mathbb{R^N\times R^N}\times(0,T)$. We know that $\frac{H(x,y,u(x,t)-u(y,t))}{|x-y|^N}\in L^1(\mathbb{R^N\times R^N}\times(0,T))$, therefore by \cref{new3}, \cref{upper1} and H\"older's inequality, we get
\begin{align}\label{new4}
\int_0^T\iint\limits_{\mathbb{R}^N\times\mathbb{ R}^N}\frac{|H(x,y,u(x,t)-u(y,t)+s(\psi(x,t)-\psi(y,t))|}{|x-y|^N}\,dx\,dy\,dt < \infty
\end{align}
for all $\psi\in C^\infty_0(\mathbb{R}^N\times\mathbb{R}^N\times(0,T))$ and uniformly in $s\in(0,1)$. 

We would like to use $v=u+s\psi$ as a comparision function in the variational inequality, however it does not have the requisite time regularity. For this purpose, we consider $v_h := [u]_h + s[\psi]_h$ where the initial value for computing $[u]_h$ is taken to be $u_0$, while for $[\psi]_h$, we choose $\psi(0) = 0$. Moreover, we must choose $(v_h)_+$ as the comparison function in the variational inequality \cref{defvar} since $v_h$ might become negative. As usual, $\de_t (v_h)_+=(\de_t v_h)_+$. Also recall that $(v_h)_+(0) = v_h(0) = u_0$. Therefore, there is no boundary term at $t=0$ in the variational inequality. Thus, we get
\begin{align}\label{varineq33}
    \mathfrak{B}[u(T),(v_h)_+(T)]&+\int_0^T\iint\limits_{\mathbb{R}^N\times\mathbb{ R}^N}\frac{H(x,y,u(x,t)-u(y,t))}{|x-y|^N}\,dx\,dy\,dt\nonumber\\
&\leq
\int_0^T\iint\limits_{\mathbb{R}^N\times\mathbb{ R}^N}\frac{H(x,y,(v_h)_+(x,t)-(v_h)_+(y,t))}{|x-y|^N}\,dx\,dy\,dt\nonumber\\
&\qquad\qquad+\int_0^T\int_{\Omega}\de_t(v_h)_+(b((v_h)_+)-b(u))\,dx\,dt\nonumber\\
&\leq
\underbrace{\int_0^T\iint\limits_{\mathbb{R}^N\times\mathbb{ R}^N}\frac{H(x,y,v_h(x,t)-v_h(y,t))}{|x-y|^N}\,dx\,dy\,dt}_{I}\nonumber\\
&\qquad\qquad+\underbrace{\iint\limits_{\Omega_T\cap\{v_h\geq 0\}}\de_t v_h(b(v_h)-b(u))\,dx\,dt}_{II},
\end{align} where the last inequality follows since the function $\xi\mapsto H(x,y,\xi)$ attains its minimum at $\xi=0$ by \cref{new2}.

By \cref{convergenceoffunctionals}, we have 
\begin{align}
\frac{H(x,y,v_h(x,t)-v_h(y,t))}{|x-y|^N}\leq \left[\frac{H(x,y,u(x,t)-u(y,t)+s(\psi(x,t)-\psi(y.t)))}{|x-y|^N}\right]_h,
\end{align} where the initial value for computing the mollification is taken to be $\frac{H(x,y,u_0(x)-u_0(y))}{|x-y|^N}$. We know from \cref{new4} that $\frac{H(x,y,(u+s\psi)(x,t)-(u+s\psi)(y,t))}{|x-y|^N}\in L^1(\mathbb{R^N\times R^N}\times(0,T))$. Hence from \cref{propertiesofmolli} (i), we receive
\begin{align}\label{limitI}
\limsup_{h\downarrow 0} I \leq \int_0^T\iint\limits_{\mathbb{R}^N\times\mathbb{ R}^N}\frac{H(x,y,u(x,t)-u(y,t)+s(\psi(x,t)-\psi(y,t)))}{|x-y|^N}\,dx\,dy\,dt.
\end{align}

We turn to the second term $II$. This part is handled exactly as in \cite{bogeleinDoublyNonlinearEquations2018}. By the definition of $v_h$, the term $II$  can be expanded to the following three terms:
\begin{align}\label{estII}
II = \underbrace{-s\iint\limits_{\Omega_T\cap\{v_h\geq 0\}}\de_t[\psi]_h b(u)\,dx\,dt}_{II_1} - \underbrace{\iint\limits_{\Omega_T\cap\{v_h\geq 0\}} \de_t[u]_h b(u)\,dx\,dt}_{II_2}+\underbrace{\iint\limits_{\Omega_T}\de_t(v_h)_+ b((v_h)_+)\,dx\,dt}_{II_3}.
\end{align}

For the integrand in $II_2$, we have the following chain of estimates
\begin{align}
-\de_t[u]_h b(u)\stackrel{\cref{generalmolliX}}{=}\frac{1}{h}([u]_h-u)\phi'(u)&\le \frac{1}{h}(\phi([u]_h)-\phi(u))\nonumber\\
&\le \frac{1}{h}([\phi(u)]_h-\phi(u))\nonumber\\
&\stackrel{\cref{generalmolliX}}{=}-\de_t[\phi(u)]_h
\end{align}
where the final inequality is a consequence of convexity of $\phi$. The mollification of $\phi(u)$ is defined in the usual way with initial condition $\phi(u_0)$. 

As a result, we receive
\begin{align}\label{estII2}
II_{2}&\leq -\iint\limits_{\Omega_T\cap\{v_h\geq 0\}} \de_t[\phi(u)]_h\,dx\,dt\nonumber\\
&=-\iint\limits_{\Omega_T} \de_t[\phi(u)]_h\,dx\,dt+\iint\limits_{\Omega_T\cap\{v_h\leq 0\}} \de_t[\phi(u)]_h\,dx\,dt\nonumber\\
&=\int_{\Omega}\phi(u_0)\,dx-\int_{\Omega}[\phi(u)]_h(T)\,dx+\iint\limits_{\Omega_T\cap\{v_h\leq 0\}} \de_t[\phi(u)]_h\,dx\,dt
\end{align}

For the integral $II_{3}$ we write
\begin{align}\label{estII3}
II_{3}\stackrel{\phi'=b}{=}\iint\limits_{\Omega_T}\de_t[\phi((v_h)_+)]\,dx\,dt=\int\limits_{\Omega}\phi(([u]_h+s[\phi]_h)_+)(T)\,dx-\int\limits_{\Omega}\phi(u_0)\,dx
\end{align}

Substituting \cref{estII2} and \cref{estII3} in \cref{estII}, we receive
\begin{align}\label{estIImore}
II\leq -s\iint\limits_{\Omega_T}\de_t[\psi]_h b(u)\,dx\,dt&+\underbrace{\iint\limits_{\Omega_T\cap\{v_h\le 0\}}[s\de_t[\psi]_h b(u)+\de_t[\phi(u)]_h]\,dx\,dt}_{III}\nonumber \\
&\qquad+ \underbrace{\int\limits_{\Omega}[\phi(([u]_h+s[\psi]_h)_+)(T)-[\phi(u)]_h](T)\,dx}_{IV}.
\end{align}

By \cref{mainthm:e}, we have $\de_t\phi(u)=2\sqrt{\phi(u)}\de_t\sqrt{\phi(u)}\in L^1(\Omega_T)$ therefore by \cref{propertiesofmolli} (vi), we conclude that $\de_t[\phi(u)]_h \rightarrow \de_t\phi(u)$ in $L^1(\Omega_T)$ as $h \downarrow 0$. 

We also have $\de_t[\psi]_h \rightarrow \de_t\psi$ uniformly on $\Omega_T$ and $\limsup_{h \downarrow 0}\chi_{\{v_h\le 0\}}\le \chi_{\{u+s\psi\le 0\}}$ almost everywhere in $\Omega_T$.

Hence, we have
\begin{align}\label{estIII}
\limsup_{h\downarrow 0} III \le
\iint\limits_{\Omega_T\cap\{u+s\psi\le 0\}}[s|\de_t\psi| |b(u)|+|\de_t\phi(u)|]\,dx\,dt
\end{align}

By \cref{propertiesofmolli} (iv), we also have that $[u]_h(T) \rightarrow u(T)$ and $[\psi]_h (T)\rightarrow 0$ pointwise a.e. on  $\Omega$. Then by convexity of $\phi$ (which leads to $\phi(v_h)\leq [\phi(u+s\psi)]_h$), \cref{propertiesofmolli} (i) and a version of dominated convergence theorem, we find that the integral $IV$ vanishes in the limit $h \downarrow 0$. With this knowledge and \cref{estIII}, we obtain 
\begin{align}\label{estIIthree}
\limsup_{h\downarrow 0}II \leq -s\iint\limits_{\Omega_T}\de_t\psi\, b(u)\,dx\,dt+\underbrace{\iint\limits_{\Omega_T\cap\{u+s\psi\le 0\}}[s|\de_t\psi| |b(u)|+|\de_t\phi(u)|]\,dx\,dt}_{V}
\end{align}

For the term $V$ in \cref{estIIthree}, the function $u$ satisfies $u\leq s\norm{\psi}_{L^\infty}$. Therefore, using monotonicity of $\phi$ and the property (4) in \cref{techlem1}, we get 
\begin{align}\label{somest}
|\de_t\phi(u)|\leq 2\sqrt{\phi(u)}|\de_t\sqrt{\phi(u)}|\leq
\sqrt{\phi(\psi)}\,s^{\frac{l+1}{2}}|\de_t\sqrt{\phi(u)}|
\end{align}
for any $s \in (0, 1)$. 

Substituting \cref{somest} in \cref{estIIthree}, we conclude 
\begin{align}\label{estIIfour}
\limsup_{h\downarrow 0} II \leq -s\iint\limits_{\Omega_T}\de_t\psi\, b(u)\,dx\,dt + c s \iint\limits_{\Omega_T\cap\{u+s\psi\le 0\}}[ |b(u)|+|\de_t\phi(u)|]\,dx\,dt,
\end{align}
where we use the fact that $l\geq 1$, $s\in (0,1)$ and the constant $c$ depends on $\psi$.

Now, on dividing inequality \cref{varineq33} by $s$ and substituting \cref{limitI} and \cref{estIIfour} in order to pass to the limit $h \downarrow 0$ while using the non-negativity of the boundary term, we receive
\begin{align}\label{best1}
\iint\limits_{\Omega_T}\de_t\psi b(u)\,dx\,dt \leq & \frac{1}{s} \int_0^T\hspace{-0.5cm}\iint\limits_{\RR^N\times\RR^N} \frac{H(x,y,(u+s\psi)(x,t)-(u+s\psi)(y,t))-H(x,y,u(x,t)-u(y,t))}{|x-y|^N}\,dx\,dy\,dt\nonumber\\
+ & c\iint\limits_{\Omega_T\cap\{u+s\psi\le 0\}}[ |b(u)|+|\de_t\phi(u)|]\,dx\,dt
\end{align}

By mean value theorem we write the first integral on the right-hand side of \cref{best1} as
\begin{align}\label{best2}
\frac{1}{s} &\int_0^T\iint\limits_{\mathbb{R}^N\times \mathbb{R}^N} \frac{H(x,y,(u+s\psi)(x,t)-(u+s\psi)(y,t))-H(x,y,u(x,t)-u(y,t))}{|x-y|^N}\,dx\,dy\,dt\nonumber\\
&=\int_0^T\iint\limits_{\mathbb{R}^N\times \mathbb{R}^N} \int_0^1 \frac{D_\xi H(x,y,(u+s\sigma\psi)(x,t)-(u+s\sigma\psi)(y,t))\cdot(\psi(x,t)-\psi(y,t))}{|x-y|^N}\,d\sigma\,dx\,dy\,dt
\end{align}

On account of \cref{upper1} and H\"older's inequality, the integrand in the above integral satisfies 
\begin{align}
    &\frac{D_\xi H(x,y,(u+s\sigma\psi)(x,t)-(u+s\sigma\psi)(y,t))\cdot(\psi(x,t)-\psi(y,t))}{|x-y|^N}\nonumber\\
    &\qquad\qquad\leq \left(\frac{|\psi(x,t)-\psi(y,t)|}{|x-y|^{\frac{N}{p}+s}}\right)F \in L^1(\mathbb{R}^N\times \mathbb{R}^N\times (0,T)),
\end{align}
independently of $s \in (0, 1)$. On passing to the limit as $s \downarrow 0$ on the right-hand side of \cref{best2}, we obtain 
\begin{align}\label{best3}
    \lim_{s\downarrow 0}\frac{1}{s} &\int_0^T\iint\limits_{\mathbb{R}^N\times \mathbb{R}^N} \frac{H(x,y,(u+s\psi)(x,t)-(u+s\psi)(y,t))-H(x,y,u(x,t)-u(y,t))}{|x-y|^N}\,dx\,dy\,dt\nonumber\\
    &=\int_0^T\iint\limits_{\mathbb{R}^N\times \mathbb{R}^N} \int_0^1 \frac{D_\xi H(x,y,u(x,t)-u(y,t))\cdot(\psi(x,t)-\psi(y,t))}{|x-y|^N}\,d\sigma\,dx\,dy\,dt
    \end{align}

We also have 
\begin{align}\label{best4}
\lim_{s\downarrow 0}\iint\limits_{\Omega_T\cap\{u+s\psi\le 0\}}[ |b(u)|+|\de_t\phi(u)|]\,dx\,dt=\iint\limits_{\Omega_T\cap\{u = 0\}}[ |b(u)|+|\de_t\phi(u)|]\,dx\,dt=0
\end{align}
since $b(0) = 0 = \phi(0)$ and $\de_t\sqrt{\phi(u)} = 0$ almost everywhere where $\sqrt{\phi(u)} = 0$.

Substituting \cref{best3} and \cref{best4} in \cref{best1}, we receive 
\begin{align*}
\iint\limits_{\Omega_T}\de_t\psi\, b(u)\,dx\,dt \leq \int_0^T\iint\limits_{\mathbb{R}^N\times \mathbb{R}^N} \frac{D_\xi H(x,y,u(x,t)-u(y,t))}{|x-y|^N}.(\psi(x,t)-\psi(y,t))\,dx\,dy\,dt.
\end{align*}

We obtain equality above by replacing $\psi$ with $-\psi$ which is what we set out to prove.
\end{proof}

\begin{remark}
	The assumption \cref{upper1} can be substituted with either of the following two conditions. It is not clear which among the three conditions is best. However, \cref{upper1} subsumes \cref{eq:bound_above_H}.
	\begin{itemize}
		\item[(H1).] For every $\psi \in C_0^\infty(\Omega_T)$ there is a function $F \in L^1(\RR^N \times \RR^N \times (0,T))$, which may depend on $u$ and $\psi$, so that $$\frac{|D_\xi H(x,y,u(x,t)-u(y,t)+s(\psi(x,t)-\psi(y,t)))|}{|x-y|^{N}} \leq F$$
		holds almost everywhere for every $0<s<1$.
		\item[(H2).] For every $\psi \in C_0^\infty(\Omega_T)$ there is a function $F \in L^1(\RR^N \times \RR^N \times (0,T))$, which may depend on $u$ and $\psi$, so that
		\begin{align*}
			\frac{|D_\xi H(x,y,u(x,t)-u(y,t)+s(\psi(x,t)-\psi(y,t)))|}{|x-y|^{N-1}} \leq F
		\end{align*}
		holds almost everywhere for every $0<s<1$
	\end{itemize}
\end{remark}

\appendix
\section{Mollification in Time}\label{Molli}
\renewcommand{\thesection}{\Alph{section}}
In the definition of variational solutions, the test functions or comparison functions have additional time regularity compared to the solutions, therefore the variational solutions themselves cannot be used as comparison functions. To fix this, we need a smoothening in time. Let $\Omega$ be an open subset of $\mathbb{R}^N$. For $T>0$, $v\in L^1(\Omega_T)$, $v_0\in L^1(\Omega)$ and $h\in (0,T]$, we define
\begin{align}
    \label{timemolli}
    [v]_h(\cdot,t)=e^{-\frac{t}{h}}v_0 + \frac{1}{h}\int_0^t e^{\frac{s-t}{h}}v(\cdot,s)\,ds,
\end{align} for $t\in [0,T]$. The basic properties of time mollification were proved earlier in~\cite{kinnunenPointwiseBehaviourSemicontinuous2006,bogeleinParabolicSystemsQGrowth2013}. We state them below for easy reference and prove the ones that are new.

\begin{proposition}(\cite[Lemma~B.1]{bogeleinParabolicSystemsQGrowth2013})\label{generalmolliX}
Let $X$ be a Banach space and assume that $v_0\in X$, and moreover $v\in L^r(0,T;X)$ for some $1\leq r\leq\infty$. Then, the mollification in time defined by~\eqref{timemolli} belongs to $L^r(0,T;X)$ and
\begin{align}
    ||[v]_h||_{L^r(0,T;X)}\leq ||v||_{L^r(0,t_0;X)}+ \left(\frac{h}{r}\left(1-e^{-\frac{t_0 r}{h}}\right)\right)^{\frac{1}{r}}||v_0||_{X},
\end{align} for any $t_0\in(0,T)$. Moreover, we have
\begin{align}
    \de_t[v]_h\in L^r(0,T;X)\mbox{ and } \de_t[v]_h=-\frac{1}{h}([v]_h-v).
\end{align}
\end{proposition}

\begin{lemma}(\cite[Lemma~2.2]{bogeleinExistenceEvolutionaryVariational2014})
\label{propertiesofmolli}
Let $\Omega$ be an open subset of $\mathbb{R}^N$. Suppose that $v\in L^1(\Omega_T)$ and $v_0\in L^1(\Omega)$. Then, the mollification in time as defined in~\eqref{timemolli} satisfies the following properties:
\begin{itemize}
    \item[(i)] Assume that $v\in L^p(\Omega_T)$ and $v_0\in L^p(\Omega)$ for some $p\geq 1$. Then, it holds true that $[v]_h\in L^p(\Omega_T)$ and the following quantitative bound holds.
    \begin{align}
    ||[v]_h||_{L^p(\Omega_T)}\leq ||v||_{L^p(\Omega_T)}+ h^{1/p}||v_0||_{L^p(\Omega)}.
\end{align}
Moreover, $[v]_h\to v$ in $L^p(\Omega_T)$ and pointwise a.e. in $\Om_T$ as $h\to 0$.
\item[(ii)] Assume that $v\in L^p(0,T;W^{s,p}(\Omega))$ and $v_0\in W^{s,p}(\Omega)$ for some $p > 1$ and $s\in (0,1]$. Then, it holds true that $[v]_h\in L^p(0,T;W^{s,p}(\Omega))$ and the following quantitative bound holds.
    \begin{align}
    ||[v]_h||_{L^p(0,T;W^{s,p}(\Omega))}\leq ||v||_{L^p(0,T;W^{s,p}(\Omega))}+ h^{1/p}||v_0||_{W^{s,p}(\Omega)}.
\end{align}
Moreover, $[v]_h\to v$ in $L^p(0,T;W^{s,p}(\Omega))$ as $h\to 0$.
\item[(iii)] Suppose that $v\in L^p(0,T;W^{s,p}_0(\Omega))$ and $v_0\in W^{s,p}_0(\Omega)$ for some $p > 1$ and $s\in (0,1]$. Then, it holds true that $[v]_h\in L^p(0,T;W^{s,p}_0(\Omega))$.
\item[(iv)] Suppose that $v\in C^0([0,T];L^2(\Omega))$ and $v_0\in L^2(\Omega)$. Then, it holds true that $[v]_h\in C^0([0,T];L^2(\Omega))$, $[v]_h(\cdot,0)=v_0$. Moreover, $[v]_h\to v$ in $C^0([0,T];L^2(\Omega))$ as $h\to 0$ and pointwise a.e. in $\Om$ for every $t\in [0,T]$.
\item[(v)] Suppose that $v\in L^\infty(0,T;L^2(\Omega))$ and $v_0\in L^2(\Omega)$. Then, it holds true that $[v]_h\in L^\infty(0,T;L^2(\Omega))$. Moreover, $[v]_h=-\frac{1}{h}([v]_h-v)$.
\item[(vi)] Let $r\geq 1$. Suppose that $\de_t v\in L^r(\Omega_T)$ then $\de_t [v]_h\to \de_t v$ in $L^r(\Omega_T)$ as $h\to 0$.
\end{itemize}
\end{lemma}

\begin{proof}
The proofs of statements $(i), (iv), (v)$ and $(vi)$ are the same as in~\cite[Lemma~B.2]{bogeleinParabolicSystemsQGrowth2013} and \cite[Lemma 6.2]{bogeleinDoublyNonlinearEquations2018}. The proofs for $(ii)$ and $(iii)$ are given in the appendix of \cite{prasadExistenceVariationalSolutions2021a}.
\end{proof}

We also note the following theorem.

\begin{theorem}
\label{convergenceoffunctionals}
Let $T>0$, and assume that $v\in L^1(0,T;W^{s,1}(\mathbb{R}^N))$ with
\begin{align*}
    \frac{H(x,y,v(x,t)-v(y,t))}{|x-y|^N}\in L^1(0,T;L^1(\mathbb{R}^N\times\mathbb{R}^N)),
\end{align*} and $v_0\in W^{s,1}(\mathbb{R}^N)$, with
\begin{align*}
    \frac{H\left(x,y,v_0(x)-v_0(y)\right)}{|x-y|^N}\in L^1(\mathbb{R}^N\times\mathbb{R}^N).
\end{align*} Then, we have
\begin{align}
    \frac{H(x,y,[v]_h(x,t)-[v]_h(y,t))}{|x-y|^N}\leq \left[\frac{H(x,y,v(x,t)-v(y,t))}{|x-y|^N}\right]_h,
    \end{align}
    so that
\begin{align*}
    \frac{H\left(x,y,[v]_h(x,t)-[v]_h(y,t)\right)}{|x-y|^N}\in L^1(0,T;L^1(\mathbb{R}^N\times\mathbb{R}^N)).
\end{align*}
Moreover,
\begin{align*}
    \lim_{h\to 0}\int_0^T\iint\limits_{\mathbb{R}^N\times\mathbb{R}^N}\frac{H\left(x,y,[v]_h(x,t)-[v]_h(y,t)\right)}{|x-y|^N}\,dx\,dy\,dt=\int_0^T\iint\limits_{\mathbb{R}^N\times\mathbb{R}^N}\frac{H\left(x,y,v(x,t)-v(y,t)\right)}{|x-y|^N}\,dx\,dy\,dt
\end{align*}
\end{theorem}

\begin{proof}
The proof is the same as that of \cite[Lemma~2.3]{bogeleinExistenceEvolutionaryVariational2014}. The only difference is we use convexity of $\xi\to H(x,y,\xi)$ and use convergence in $L^1(0,T;L^1(\mathbb{R}^N\times\mathbb{R}^N))$.
\end{proof}

\section*{Acknowledgments} 
The authors would like to thank Karthik Adimurthi for suggesting this problem and helpful discussions. The authors were supported by the Department of Atomic Energy,  Government of India, under project no.  12-R\&D-TFR-5.01-0520. 

\bibliography{MyLibrary}
%\bibliographystyle{plain}
%\begin{thebibliography}{}
%\bibitem{}
%\label{bib:XYZ}
%\end{thebibliography}

\end{document}